\documentclass[twoside,11pt]{amsart}
\usepackage[all]{xy}
\CompileMatrices
\usepackage{amsfonts}
\usepackage{amssymb}
\usepackage{amsthm}
\usepackage{amsmath}
\usepackage{ifthen}
\usepackage{url}
\usepackage[backref]{hyperref}
\usepackage[usenames]{color}

 \newtheorem{thm}{Theorem}[section]
 \newtheorem{prop}[thm]{Proposition}
 \newtheorem{cor}[thm]{Corollary}
 \newtheorem{lem}[thm]{Lemma}
  
 \newtheorem{conj}[thm]{Conjecture}

\theoremstyle{definition}

\theoremstyle{remark}
\newtheorem{rem}[thm]{Remark}

  1

\renewcommand{\L}{\ifmmode {\mathcal{L}}\else$\mathcal{L}$\ \fi}

\newcommand{\bbC}{\ifmmode {\mathbb{C}}\else$\mathbb{C}$\ \fi}
\newcommand{\bbR}{\ifmmode {\mathbb{R}}\else$\mathbb{R}$\ \fi}

\newcommand{\be}{\begin{equation}}
\newcommand{\ee}{\end{equation}}

\newcommand{\fpbar}{\ifmmode {\overline{\mathbb{F}_p}}\else$\mathbb{F}_p$\ \fi}

\newcommand{\fp}{\ifmmode {\mathbb{F}_p}\else$\mathbb{F}_p$\ \fi}
\newcommand{\zp}{\ifmmode {\mathbb{Z}_p}\else$\mathbb{Z}_p$\ \fi}

\newcommand{\z}{\mathbb{Z}}
\newcommand{\Z}{\mathbb{Z}}
\newcommand{\zpMod}{\ifmmode\mbox{$\zp$-Mod}\else$\zp$-Mod \fi}
\newcommand{\Mod}{\ifmmode\mbox{$\Lambda$-Mod}\else$\Lambda$-Mod \fi}
\renewcommand{\mod}{\ifmmode\mbox{$\Lambda$-mod}\else$\Lambda$-mod
\fi}
\newcommand{\La}{\ifmmode\Lambda\else$\Lambda$\fi}

\newcommand{\Hom}{{\mathrm{Hom}}}

\renewcommand{\H}{\mathrm{H}}

\newcommand{\M}{\ifmmode {\frak M}\else${\frak M}$ \fi}
\newcommand{\m}{\ifmmode {\frak m}\else$\frak m$ \fi}
\newcommand{\mh}{\ifmmode {\frak m}(H)\else${\frak m}(H)$ \fi}
\newcommand{\p}{\ifmmode {\frak p}\else${\frak p}$\ \fi}
\renewcommand{\P}{\ifmmode {\frak P}\else${\frak P}$\ \fi}
\newcommand{\Zp}{\mbox{${\Bbb Z}_p$}}
\newcommand{\e}{\ifmmode {\mathcal{E}}\else$\mathcal{E}$ \fi}

\renewcommand{\O}{\mathcal{ O}}
\newcommand{\G}{\ifmmode {\mathcal{G}}\else${\mathcal{G}}$\ \fi}
\newcommand{\A}{\ifmmode {\mathcal{A}}\else${\mathcal{ A}}$\ \fi}

\renewcommand{\projlim}[1] {{\lim\limits_{\stackrel{\displaystyle
\longleftarrow}{#1}}}}

\newcommand{\dirlim}[1]
{{\lim\limits_{\stackrel{\displaystyle \longrightarrow}{#1}}}}

\newcommand{\kl}{[\![}
\newcommand{\kr}{]\!]}
\newcommand{\Qp}{\ifmmode {{\Bbb Q}_p}\else${\Bbb Q}_p$\ \fi}
\newcommand{\qp}{\ifmmode {{\Bbb Q}_p}\else${\Bbb Q}_p$\ \fi}
\newcommand{\ql}{\ifmmode {{\Bbb Q}_l}\else${\Bbb Q}_l$\ \fi}
\newcommand{\Q}{\ifmmode {\Bbb Q}\else${\Bbb Q}$\ \fi}
\newcommand{\q}{\ifmmode {\Bbb Q}\else${\Bbb Q}$\ \fi}

\newcommand{\Coind}{\mbox{\mbox{\rm Coind}}}
\newcommand{\Ind}{\mathrm{Ind}}

\newcommand{\Co}{\mbox{$\Bbb C$}}

\def\sectionnam{\@empty}
\def\subsectionnam{\@empty}

\begin{document}

\title[Noncommutative MC for CM-elliptic curves]{On the non-commutative Main Conjecture for elliptic curves with complex multiplication}%

\author{Thanasis Bouganis}
\address{Universit\"{a}t Heidelberg\\ Mathematisches Institut\\
Im Neuenheimer Feld 288\\ 69120 Heidelberg, Germany.}
\email{bouganis@mathi.uni-heidelberg.de}

\author{Otmar Venjakob}%
\address{Universit\"{a}t Heidelberg\\ Mathematisches Institut\\
Im Neuenheimer Feld 288\\ 69120 Heidelberg, Germany.}
\email{venjakob@mathi.uni-heidelberg.de}
\urladdr{http://www.mathi.uni-heidelberg.de/\textasciitilde otmar/}
\thanks{We acknowledge  support by the ERC and DFG}



\date{\today}%
\maketitle \thispagestyle{empty}

In \cite{cfksv} a non-commutative Iwasawa Main Conjecture for
elliptic curves over $\mathbb{Q}$ has been formulated. In this note
we show that it holds for all CM-elliptic curves $E$ defined over
$\mathbb{Q}$. This was claimed in (loc.\ cit.) without proof, which
we want to provide now assuming that the torsion conjecture holds in
this case. Based on this we show firstly the existence of the
(non-commutative) $p$-adic $L$-function of $E$ and secondly that the
(non-commutative) Main Conjecture follows from the existence of the
Katz-measure, the work of Yager and Rubin's proof of the 2-variable
main conjecture. The main issues are the comparison of the involved
periods and to show that the (non-commutative) $p$-adic $L$-function
is defined over the conjectured in (loc.\ cit.) coefficient ring.
Moreover we generalize our considerations to the case of CM-elliptic
cusp forms.

{\em Acknowledgements:} We are grateful to John Coates and Sujatha for their interest and various
discussions.

\section{The non-commutative Main Conjecture }\label{MC}

Let $E$ be an elliptic curve defined over $\mathbb{Q}$ and $p\geq 5$
a prime at which $E$ has good ordinary reduction. Then the $p$-adic
Lie extension $F_\infty:=\Q(E(p))$ with Galois group
$\mathcal{G}=G(F_\infty/\Q)$ contains the cyclotomic
$\Z_p$-extension $\Q_{cyc}$ of $\Q$ and the quotient of
$\mathcal{G}$ by its closed normal subgroup $\mathcal{H}:=
G(F_\infty/\Q_{cyc})$ equals $\Gamma:=G(\Q_{cyc}/\Q)\cong\Zp.$

We write $D$ for the ring of integers $\mathcal{O}_L$, where $L$ is
either a finite extension of $\Qp$ or of the completion
$\widehat{\Q_p^{nr}}$  of the maximal unramified extension
$\Q_p^{nr}$ of $\Q_p.$ As usual we write

\[\Lambda:= \Lambda_D(\mathcal{G}):=D\kl \mathcal{G} \kr \]

for the Iwasawa algebra of $\mathcal{G}$ with coefficients in $D.$
Note that is a Noetherian pseudo-compact semi-local ring (which is
compact if $L$ is finite over $\qp$). We denote by
\[\frak{M}_\mathcal{H}(\mathcal{G}):=\frak{M}_{D,\mathcal{H}}(\mathcal{G})\]
the category of all finitely $\Lambda(\G)$-modules $M$ such that its
quotient $M/M(p)$ by its $p$-primary subgroup $M(p)$ is finitely
generated over the subalgebra $\Lambda(\mathcal{H})$ of
$\Lambda(\G).$ Also we recall from \cite{cfksv} the definition of
the multiplicatively closed subsets

\[S:=\{\lambda\in\Lambda |\; \Lambda/\Lambda \lambda \mbox{ is finitely generated over } \Lambda(\mathcal{H})\}\]
and
\[S^*=\bigcup_{n \geq 0}p^nS.\]

The following theorem compromises the technical heart of
\cite{cfksv}.

\begin{thm}
The sets $S$ and $S^*$ are (left and right) Ore sets, i.e.\ the
localisations $\Lambda_S$ and $\Lambda_{S^*}$ of $\Lambda$ exist and
the following holds:

\begin{enumerate}
\item The category of all finitely generated $S^*$-torsion $\La(\G)$-modules  coincides with $\frak{M}_\mathcal{H}(\mathcal{G}).$
\item There is an long exact localisation sequence of $K$-groups
\[\xymatrix{
    & {\Lambda(\G)_{S^*}} \ar@{->>}[d]_{ }   &   &   \\
  K_1( \La(\G))\ar[r]^{ } & K_1(\La(\G)_{S^*}) \ar[r]^{\partial} & K_0(\mathfrak{M}_\mathcal{H}(\G) \ar[r] & 0   }\]
  and analogously for $\La(\G)_S$ and the category of finitely generated $S$-torsion modules.
  \item There is a canonical way of evaluation an element $f\in K_1(\La(\G)_{S^*})$ at any
  continuous   representation $\rho:\G\to GL_n(\O)$ with $n\geq 1$ and $\O$ the ring of integers of a finite
  extension (depending of $\rho$) of $\Qp:$
  \[f(\rho)\in \mathbb{C}_p\cup \{\infty\},\]
  i.e.\ $f$ can be considered as a map on the set of such representations.
\end{enumerate}
\end{thm}

\begin{proof}
See \cite{ven-habil, cfksv}. Another proof that $S$ is an Ore-set
was given by Ardakov and Brown in \cite{ArdBrown} (and Schneider).
The pseudo-compact case (generalised to certain skew power series
rings) is covered by \cite{sch-ven2} or can be proven in the same
way as in \cite{cfksv}.
\end{proof}

By $S(E/F_\infty)$ and $X=X(E/F_\infty)=S(E/F_\infty)^\vee$ we
denote the (classical) $p$-primary Selmer group of $E$ over
$F_\infty$ and its Pontryagin dual, respectively. It is easy to see
that $X$ is a finitely generated $\La_\zp(\G)$-module with the
natural Galois-action on $S(E/F_\infty),$ but the following
torsion-property can be interpreted as a generalization of a deep
conjecture of Mazur.

\begin{conj}[Torsion-conjecture]\label{torConj}

The dual of the Selmer group is $S^*$-torsion:
\[X(E/F_\infty)\in\mathfrak{M}_\mathcal{H}(\G).\]
\end{conj}

Now let $K(F_{\infty})$ be the maximal abelian extension of
$\mathbb{Q}$ inside $ F_\infty$   in which $ p$   does not ramify
and $L = K (F_{\infty})_{\frak{P}}$ its completion at some
$\frak{P}$ lying over $p.$ Note that $L$ is a finite extension of
$\Qp.$ Finally we put $D = \mathcal{O}_L.$

If $\omega$  denotes the Neron differential of $E,$ we obtain the
usual real and complex periods
  $\Omega_{\pm} = \mathop{\int}\limits_{\gamma^{\pm}} \omega$ by integrating along pathes $\gamma^{\pm}$ which
  generate $H_1(E(\Co),\Z)^{\pm}.$ We set   $R = \{q \,\, {\mbox{prime}}, | j(E) |_q > 1 \}\cup \{p\}$
  and let $u,w$ be the roots of the characteristic polynomial of the action of Frobenius on the Tate-module
  $T_pE$ of $E$
 \[1- a_pT + pT^2 = (1 - uT) (1-wT),\;\; u \in \mathbb{Z}_{p}^{\times}.\]
  Further let   $p^{{\frak{f}}_p (\rho)}$ be the  $ p$-part of the conductor of an Artin representation $\rho,$ while
    $P_p (\rho, T) = \det (1 - Frob_p^{-1} \,\, T
  | V_{\rho}^{I_p})$   describes the Euler-factor of $\rho$ at $p.$ We also
  set   $d^{\pm} (\rho) =  \dim_\mathbb{C} V_{\rho}^{\pm}$ and denote by $\mathop{\rho}\limits^{\vee}$ the contragredient representation of $\rho.$
By $e_p(\rho)  $ we denote the local $\epsilon$-factor of $\rho$ at
  $p. $ In the notation of \cite{Tate} this is
  $e_p(\rho,\psi(-x),dx_1)$ where $\psi$ is the additive character
  of $\Q_p$ defined by $x \rightarrow exp(2\pi i x)$ and $dx_1$ is
  the Haar measure that gives volume 1 to $\zp$.
  Finally, in order to express special values of complex $L$-functions in the $p$-adic world,
  we fix embeddings   of  $ \bar{\mathbb{Q}} $ into $ \mathbb{C}$ and $\mathbb{C}_p,$ the completion of an algebraic closure of $\Qp. $

\begin{conj}[Existence of $p$-adic $L$-function]\label{Conj2}
There is a (unique) $ \mathcal{L}_E \in K_1(\Lambda (\G)_{
S^{\normalsize  *}}) $ such that
\[
\mathcal{L}_E (\mathop{\rho}\limits^{\vee}) = \frac{L_R(E, \rho,
1)}{\Omega_+^{d^+ (\rho)} \Omega_-^{d^-(\rho)}}\, e_p(\check{\rho})
\frac{P_p(\mathop{\rho}\limits^{\vee}, u^{-1})} {P_p(\rho, w^{-1})}
u^{-{\frak{f}}_p (\check{\rho})}
\]

for all Artin representations $\rho$ of $\G.$
\end{conj}

For a class $[M]\in K_0(\mathfrak{M}_\mathcal{H}(\G))$ we denote by
$[M]_D$ the base change to $K_0(\mathfrak{M}_{D, \mathcal{H}}(\G)).$

\begin{conj}[Main Conjecture]\label{Conj3}
The $p$-adic $L$-function $\mathcal{L}_E$ is a characteristic
element of $X(E/F_\infty):$
\[\partial \mathcal{L}_E=[X(E/F_\infty)]_D.\]
\end{conj}

We refer the reader to \cite{BV1,BV2} where a refined version
involving leading terms is discussed and where some implications of
the Main Conjecture are explained.

\section{The CM-case}

In this section we assume that $E$ is still defined over $\Q,$ has
conductor $N:=N_E$ and moreover admits complex multiplication by the
ring of integers $\O_K$ of a quadratic imaginary field
$K=\Q(\sqrt{-d_K}),$ where cf.\ Appendix A {\S}3 of \cite{silv2} the
discriminant $d_K>0$ lies in the finite set
\[\{3,4,7,8,11,19,43,67,163\}.\]


Note that since $j(E)\in\Q$ the class number of $K$ is trivial:
$h_K=1.$ As we assume that $E$ has good ordinary reduction at our
fixed prime $p,$ the latter decomposes into two primes
$\mathfrak{p}=(\pi)$ and $\bar{\mathfrak{p}}=(\bar{\pi}),$ such that
$\pi$ and $\bar{\pi}$ are conjugate under $c\in\Delta:=G(K/\Q)$
which is induced by complex conjugation. We fix an embedding of $K$
into $\bar{\Q}$ such that $\mathfrak{p}$ corresponds to the induced
embedding of $K$ into $\mathbb{C}_p$ (using the convention of the
previous section).

Setting  $H:=G(K_\infty/K_{cyc}),$ $G:=(K_\infty/K)$  we obtain
isomorphisms $G/H\cong\Gamma$ and $\mathcal{H}/H\cong\Delta\cong
\mathcal{G}/G.$ The situation can be illustrated by the following
diagram of field extensions:

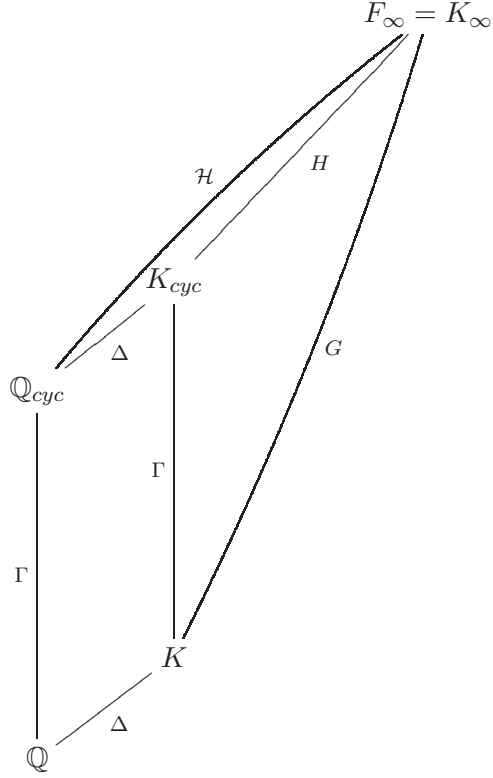
\[\xymatrix{
    &   &  & F_\infty=K_\infty     \\
  &   &   &   \\
    &   &   &   \\
     & K_{cyc} \ar@{-}[dddd]_{\Gamma} \ar@{-}[rruuu]_{H} &   &   \\
  {\Q_{cyc}} \ar@{-}[dddd]_{\Gamma} \ar@{-}[ru]_{\Delta}\ar@/^/@{-}[rrruuuu]^{\mathcal{H}} &   &   &   \\
     &   &   &   \\
    &   &   &   \\
    & K   \ar@/_/@{-}[rruuuuuuu]_{G}  &   &   \\
  {\Q} \ar@{-}[ru]_{\Delta } &  &   &    }\]

In particular the non-commutative group $\G$ is an extension of the
cyclic group of order $2$ by the abelian group $G,$ i.e.\ very close
to being abelian.

One strategy to verify the Main Conjecture for $E$ in the previous
section would be a close analysis how the groups
$K_1(\La(\G)_{S^*})$ and $K_0(\mathfrak{M}_\mathcal{H}(\G))$ look
like, in order to construct the (non-commutative) $p$-adic
$L$-function and to know how to show that two classes in
$K_0(\mathfrak{M}_\mathcal{H}(\G))$ in the absence of a suitable
structure theory. This certainly can be done, but we will follow a
much simpler way.

In this work we are going to make the following assumption:

\textbf{Assumption:} The Conjecture \ref{torConj} in the CM case is
true.

\begin{rem}\label{torConjThm}
We remark that this assumption implies that $X(E/F_{\infty})$ is
actually $S$-torsion. For this one has simply to show that the
$\mu$-invariant of $X(E/F_{\infty})$ is trivial since it is known by
\cite[th\'{e}or\`{e}me 2.4]{perriou} that $X(E/F_{\infty})$ contains
no non-trivial pseudonull $\Lambda(G)$ submodules. But the
triviality of the $\mu$-invariant follows from the results of
Schneps \cite{schneps} after employing a descent argument similar to
the one done by Hatchimori and Venjakob in \cite[theorem
5.3]{HatchVenj}.

\end{rem}

The key ingredient for our proof of the Main Conjecture is the fact
that to $E$ there is attached a $A_0$- Gr\"{o}{\ss}encharacter, in
the sense of Weil, $\psi$ of weight $(1,0)$ and conductor
$\mathfrak{f}=\mathfrak{f}_\psi$ with $N = d_K
N_{K/\Q}(\mathfrak{f})$ such that \be \label{induction} "E=\Ind_\Q^K
\psi",\ee This can either be understood

\begin{enumerate}
\item in the sense of the attached compatible system of $l$-adic
Galois representations. Indeed to the character $\psi$ there is by
Weil attached a compatible system of $\ell$-adic representations and
then the corresponding system of $\ell$-adic representations of $E$
are obtained by induction from the absolute Galois group of $K$ to
the absolute Galois group of $\Q$.
\item or in the automorphic induction sense. Indeed the automorphic
counterpart of $\psi$ is an adelic character $\psi :
\mathbb{A}_K^{\times}/K^{\times} \rightarrow \mathbb{C}^{\times}$
and the corresponding cuspidal automorphic representation $\pi_E$ of
$GL_2(\mathbb{A}_\Q)$ associated to $E$ is obtained by automorphic
induction from $GL_1(\mathbb{A}_K)$ to $GL_2(\mathbb{A}_\Q)$.
\end{enumerate}
There is also a Gr\"{o}{\ss}encharacter $\bar{\psi}$ of weight
$(0,1)$ (attached to $E$). As it was first shown by Weil these
characters although being adelic in nature can also be interpreted
as Galois characters of $G(K(\mathfrak{f}p^\infty)/K)$ or
$G(K_\infty/K)$ with values in $\mathbb{Z}_p^*,$ see also \cite[page
38]{deSh}. Note that, since $E$ is defined over $\Q,$ it holds by
\cite[page 559]{katz} that

\be   \bar{\psi}=\psi^c\ee

where $\psi^c(g)=\psi(cgc^{-1})$ for all $g\in G(K_\infty/K).$

By the functorial  properties of $K$- and representation theory and
\eqref{induction} the desired result can be reduced to the
two-variable Main Conjecture as proved by Rubin. In order to state
it we introduce $\mathcal{X}(K_\infty)$ to be the Galois group of
the maximal abelian $p$-extension of $K_\infty$ which is unramified
outside $\mathfrak{p}.$  This is a finitely generated torsion
$\La(G)$-module, \cite[page 37]{rubin}. Furthermore, let
\[T_\pi:=T_\pi E:=\projlim{n} E[\pi^n]\] be the $\pi$-primary
Tate-module of $E,$ similarly for $\bar{\pi},$  and
\[T_\pi^*=\Hom(T_\pi,\zp)\] its $\Zp$-dual representation of $G_K.$
Note that the action of $G_K$ on $T_\pi$ and $T_{\bar{\pi}}$ is
given by the characters $\psi$ and $\bar{\psi},$ respectively. Our
above philosophy is confirmed by the following

\begin{prop}\label{Selmerind}
There is a natural isomorphism of $\La(\G)$-modules
\[X(E/K_\infty)\cong \Ind_\G^G\left(
\mathcal{X}(K_\infty)\otimes_\zp T_\pi^*\right),\] where $G$ acts
diagonally on the tensor product $\mathcal{X}(K_\infty)\otimes_\zp
T_\pi^* \cong \mathcal{X}(K_\infty)(\psi^{-1}) .$
\end{prop}

The proof is a modification of the proof of the old observation by
John Coates that the $\pi$-Selmer group $S_\pi(E/K_\infty)$ of $E$
over $K_\infty$ is canonically isomorphic to
\[S_\pi(E/K_\infty)\cong\Hom(\mathcal{X}(K_\infty), E(\pi)).\]

\begin{proof}
Conferring to \cite[page 124]{deSh} the $\pi$-primary Selmer group
is given by the following
 exact sequence
 \[\xymatrix@C=0.5cm{
   0 \ar[r] & S_\pi(E/K_\infty) \ar[rr]^{ } && {\H^1(G_{S_p}(K_\infty),E(\pi)) }\ar[rr]^{ } && {\Coind_G^{G_{\bar{\mathfrak{p}}}} \H^1(K_{\infty,\nu},E)(\pi)}     }\]
where the coinduction is dual to the induction functor, $\nu$
denotes a prime of $K_\infty$ above $\bar{\mathfrak{p}}.$ Applying
the exact functor $\Coind_\G^G,$ using the transitivity of
coinduction, the isomorphisms
\[\Coind_\G^G\H^1(G_{S_p}(K_\infty),E(\pi))\cong\Hom(G_{S_p}(K_\infty),\Coind_\G^G E(\pi))\cong H^1(G_{S_p}(K_\infty),E(p))\]
and
\[\H^1(K_{\infty,\nu},E)(p)\cong\H^1(K_{\infty,\nu},E)(\pi)\oplus\H^1(K_{\infty,\nu},E)(\bar{\pi}),\]
as well as the vanishing (see loc.\ cit.) of
$\H^1(K_{\infty,\nu},E)(\bar{\pi}),$ one just obtains the defining
sequence of the full $p$-primary Selmer group
\[\xymatrix@C=0.5cm{
   0 \ar[r] & S(E/K_\infty) \ar[rr]^{ } && {\H^1(G_{S_p}(K_\infty),E(p)) }\ar[rr]^{ } && {\Coind_\G^{\G_{\bar{\mathfrak{p}}}} \H^1(K_{\infty,\nu},E)(p)}.     }\]
The result now follows by taking duals.
\end{proof}

In order to introduce the (commutative) two-variable $p$-adic
$L$-function we choose a complex period $  {\Omega_\infty} \in
\mathbb{C}^\times $ such that $  \Lambda_E =
{\Omega_\infty}\mathfrak{f}_\psi$ for the period lattice $\Lambda_E$
attached to the pair $(E,\omega).$ Here we view $\mathfrak{f}_\psi$
as a lattice in $\mathbb{C}$ with respect the chosen embedding
$K\hookrightarrow \mathbb{C}$ that correspond to the type of $\psi$.
This determines the $p$-adic period
 $  \Omega_p \in \widehat{\mathbb{Z}_{p}^{nr}}^{\times} $ by the procedure  described in   \cite[page 66]{deSh}  such that the ratio $\frac{\Omega_\infty}{\Omega_p}$ is independent of any choices.
Actually we can pin down the $p$-adic $\Omega_p$ up to elements in
$\Z_p^\times$ by the rule $\frac{\Omega_p^{\Phi}}{\Omega_p}=u$ where
$\Phi$ is the extension of Frobenious to
$\widehat{\mathbb{Z}_{p}^{nr}}$. Then we have,

\begin{thm}[Manin-Vi\v{s}ik, Katz, Yager, de Shalit]\label{measureKatz} Let
$\mathfrak{f}$ be an integral ideal of $K$ relative prime to $p$.
For $D=\widehat{\mathbb{Z}_p^{nr}}$ there is a unique measure
$\mu(\mathfrak{f}\bar{\mathfrak{p}}^\infty)\in
\La_D(G(K(\mathfrak{f}p^\infty)/K))$ such that for any
gr\"{o}ssencharacter $\epsilon$ of conductor dividing
$\mathfrak{f}p^{\infty}$ and of type $(k,j)$ with $0 \leq -j$ and $k
> 0$ we have,
\[
\Omega_p^{j-k}\int_{G(K(\mathfrak{f}p^\infty)/K)}\epsilon^{-1}(\sigma)\,\,\mu(\mathfrak{f}\bar{\mathfrak{p}}^\infty)=\Omega_{\infty}^{j-k}\Gamma(k)i^k(\frac{\sqrt{d_K}}{2\pi})^{j}G(\epsilon)(1-\frac{\epsilon(\mathfrak{p})}{p})L_{\bar{\mathfrak{p}}\mathfrak{f}}(\epsilon^{-1},0)
\]
where
\[
G(\epsilon):=\frac{\phi^{k}\bar{\phi}^{j}(\mathfrak{p}^{n})}{p^{n}}G(\chi):=
\frac{\phi^{k}\bar{\phi}^{j}(\mathfrak{p}^{n})}{p^{n}}\sum_{\gamma
\in M}\chi(\gamma)(\zeta_n^{\gamma})^{-1}
\]
Here we write $\epsilon=\phi^{k}\bar{\phi}^{j}\chi$ with $\phi$ a
Gr\"{o}ssencharacter of conductor prime to $\mathfrak{p}$ and of
type $(1,0)$ and $\chi$ a character of finite order of conductor at
$\mathfrak{p}$ equal to $\mathfrak{p}^n$ and
\[
M:= \{\gamma \in
Gal(K(\mathfrak{f}\bar{\mathfrak{p}}^{\infty}\mathfrak{p}^{n})/K)|
\gamma_{|F'}=(\mathfrak{p}^{n},F'/K)\},\,\,\,\,F':=K(\mathfrak{f}\bar{\mathfrak{p}}^{\infty})
\]
\end{thm}
\begin{proof}
This theorem is from \cite[thm. II 4.14]{deSh} (where the
restriction on the type of the character is that $0 \leq -j < k$) in
combination with corollary 6.7 of the same book. We note here that
the relation between the Galois counterpart of $\epsilon$ that
appear in the left hand side of the above formula and its
automorphic counterpart that shows up in the right side are related
by
$\epsilon^{Galois}(Frob^{-1}_{\mathfrak{q}})=\epsilon^{aut}(\mathfrak{q})$,
different from deShalit who sets
$\epsilon^{Galois}(Frob_{\mathfrak{q}})=\epsilon^{aut}(\mathfrak{q})$.
This explains the difference of the formula above with the one of
deShalit in page 80 (the argument of the integral is inverted). We
also need to remark that in the theorem of de Shalit also the Gamma
factor of the $L$ function appears which is equal to $(2\pi)^{-1}$.
However in de Shalit the archimedean period used is equal to
$(2\pi)^{-1}\Omega_{\infty}$ (compare theorem 4.11 and 4.12 in
\cite{deSh}). Finally the factor $i^k$ is coming from the
normalization of the $p$-adic periods (see \cite[page 70]{deSh}).
\end{proof}

For our purposes it is important to understand the relation of the
"Gauss sum" like factor with the epsilon factor at $\mathfrak{p}$.
To this end we have the following lemma,
\begin{lem} Let $\delta \in \zp^{\times}$ be the image of $\sqrt{-d_{K}}$
under $K\hookrightarrow K_{\mathfrak{p}} = \Q_{p}$ and
$\sigma_{\delta} \in
Gal(K(\mathfrak{f}p^{\infty})/K(\mathfrak{f}\bar{\mathfrak{p}}^{\infty}))$
such that $\sigma_{\delta}(\zeta) = \zeta^{\delta}$ for all
$p$-power roots of unity. Then with notation as in the theorem above
we have
\[
G(\chi)=\chi(\sigma_{\delta})e_{\mathfrak{p}}(\chi)
\]
\end{lem}
\begin{proof}
We base the following proof on the article of Tate \cite{Tate} and
\cite{deSh} section 6.3. By \cite{deSh} (page 92) we have
\[
e_{\mathfrak{p}}(\epsilon,\psi,dx_1)^{-1}=p^{\frac{n}{2}(k+j+1)}\delta^{-k}\epsilon(\sigma_{\delta})G(\epsilon^{-1})
\]
From the definition of $\delta \in \zp^{\times}$ we have
$\bar{\phi}(\sigma_{\delta})=1$ and $\phi(\sigma_{\delta})=\delta$
(see also \cite{deSh} (page 92). In the above equation we set
$\epsilon= \bar{\chi}\phi^{-1}$ where $\phi$ as always a
Gr\"{o}ssencharacter of type $(1,0)$. In particular we have that
$\epsilon$ is a character of type $(-1,0)$. The above equation reads
(with $k=-1$ and $j=0$)
\[
e_{\mathfrak{p}}(\bar{\chi}\phi^{-1},\psi(x),dx_1)^{-1}=\delta
\bar{\chi}(\sigma_{\delta})\phi^{-1}(\sigma_{\delta})G(\chi\phi)
\]
Now we note that $\phi^{-1}(\sigma_{\delta})=\delta^{-1}$ and since
$\phi$ is unramified at $\mathfrak{p}$ we have that (see
\cite{Tate}) $e_{\mathfrak{p}}(\phi^{-1} \bar{\chi},\psi(x),dx_1) =
\phi(\mathfrak{p})^{-n}e_{\mathfrak{p}}(\bar{\chi},\psi(x),dx_1)$.
With these remarks the above equation simplifies to
\[
e_{\mathfrak{p}}(\bar{\chi},\psi(x),dx_1)^{-1}\phi(\mathfrak{p})^n=
\bar{\chi}(\sigma_{\delta})G(\chi\phi)
\]

Using now the duality property (see \cite{Tate})
\[
e_{\mathfrak{p}}(\bar{\chi},\psi(x),dx_1)e_{\mathfrak{p}}(\chi,\psi(-x),dx_1)=p^n
\]
we obtain that
\[
e_{\mathfrak{p}}(\chi,\psi(-x),dx_1)p^{-n}\phi(\mathfrak{p})^n=
\bar{\chi}(\sigma_{\delta})G(\chi\phi)
\]
But $G(\chi\phi)=\frac{\phi(\mathfrak{p})^n}{p^n}G(\chi)$ from which
we conclude that
\[
G(\chi)=\chi(\sigma_{\delta})e_{\mathfrak{p}}(\chi,\psi(-x),dx_1)=\chi(\sigma_{\delta})e_{\mathfrak{p}}(\chi)
\]
following our convention to write $e_{\mathfrak{p}}(\cdot)$ for
$e_{\mathfrak{p}}(\cdot,\psi(-x),dx_1)$.
\end{proof}

\begin{cor}
 There is a unique $\mathcal{L}'_{\bar\psi} : = \mu \in \La_D (G)$ such that
\[
 \mathcal{L}'_{\bar\psi} { (\bar\chi) } = \mathop{\int}\limits_{G}
 { \chi } d\mu = \frac{L ({\bar\psi} \chi, 1)}{{\Omega_\infty}} \,
 e_{\mathfrak{p}}(\bar{\chi})
 P_{\frak{p}}({\bar\chi}, u^{-1}) P_{\bar{\frak{p}}} ({ \chi},
 u^{-1})\,\,u^{-\mathfrak{f_{\mathfrak{p}}}(\bar{\chi})}
 \]
 for all Artin-character $\chi$ of $G.$
\end{cor}

\begin{proof} This is in principle the measure mentioned in the
above theorem \ref{measureKatz} twisting it by the fixed character
$\psi$. However in the theorem above it is the imprimitive $L$
function that appears (as we remove the Euler factors at
$\mathfrak{f}$ even if the conductor of the character is not
$\mathfrak{f}$). For our purposes we need the primitive $L$
functions. We now explain how it can be constructed. We consider the
decomposition $Gal(K_\infty/K) = G \cong \Delta \times G_1$ where
$\Delta \cong (\mathbb{Z}/p\mathbb{Z})^{\times} \times
(\mathbb{Z}/p\mathbb{Z})^{\times}$ and $G_1 \cong \mathbb{Z}^2_p$.
This decomposition induces a decomposition of Iwasawa algebras
\[
\La_D [[G]]=D[[G]] \cong D[\Delta][[G_1]] \cong \oplus_{\theta \in
\hat{\Delta}} D[[G_1]]
\]
where $\hat{\Delta} = Hom(\Delta, D^\times)$ the dual group of
$\Delta$. For a character $\chi$ of $G$ we consider its
decomposition to $\chi=\chi_{\Delta}\chi_1$ according to the above
decomposition of $G$. Similarly we have $\psi = \psi_\Delta \psi_1$.
We are going to define the element $\mu \in \La_D [[G]]$ claimed in
the proposition by defining the elements in $ D[[G_1]]$ in the above
decomposition of the Iwasawa algebras. We define the $\theta^{th}$
component $\mu_{\theta}$ as follows. We write $\mathfrak{g}$ for the
prime to $p$ part of the conductor of the character $\theta
\psi_{\Delta}$. We note that $\mathfrak{g}$ can be different from
$\mathfrak{f}_{\psi}$. We consider the measure $\delta
\,\sigma^{-1}_{\delta} \star
\Omega_p^{-1}\mu(\mathfrak{g}\bar{\mathfrak{p}}^\infty)\in
\La_D(G(K(\mathfrak{g}p^\infty)/K))$ (here $\star$ denotes
convolution of measures) and define the measure $\mu_{\theta}$ on
$G_1$ as
\[ \int_{G_1} fd\mu_{\theta}=\frac{1}{\Omega_p}\int_{G(K(\mathfrak{g}p^\infty)/K)}\psi^{-1} \theta \cdot (f\circ pr)\; d(\delta \,\sigma^{-1}_{\delta} \star\mu(\mathfrak{g}\bar{\mathfrak{p}}^\infty))\]
for all integrable functions $f$ on $G_1$. Here
$pr:G(K(\mathfrak{g}p^\infty)/K) \twoheadrightarrow
 G_1$ denotes the natural projection map and $\psi\theta$ is seen as a character of $G(K(\mathfrak{g}p^\infty)/K)$. We define the measure
$\mu$ by putting together the components $\mu_{\theta}$ according to
the isomorphism $\La_D [[G]] \cong \oplus_{\theta \in \hat{\Delta}}
D[[G_1]]$. Now the result follows from the above theorem by taking
 $k=1, j=0, \epsilon=\psi\bar{\chi}$ and noting that
 $\mathfrak{f}_\mathfrak{p}(\bar{\chi}\psi)=\mathfrak{f}_\mathfrak{p}(\bar{\chi})=\mathfrak{f}_{\mathfrak{p}}( {\chi}
 )$ because $\psi$ is unramified at $\mathfrak{p}.$ Also note that from the equation
 \[(1-uX)(1-wX)=1-a_pX + pX^2=(1-\psi(\mathfrak{p})X)(1-\psi(\bar{\mathfrak{p}})X)\] and the
 condition that $u$ is a unit in $\zp,$ it follows that
 \be u=\psi(\bar{\mathfrak{p}})=\bar{\pi} \ee
 and
 \be u^{-1}=\frac{w}{p}=\frac{\psi( {\mathfrak{p}})}{p}.\ee

\end{proof}

\begin{cor}\label{measure1}
 There is a unique $\mathcal{L}_{\bar\psi} : = \mu \in \La_D (G)$ such that
\[
 \mathcal{L}_{\bar\psi} { (\bar\chi) } = \mathop{\int}\limits_{G}
 { \bar\chi } d\mu = \frac{L ({\bar\psi} \chi, 1)}{{\Omega_\infty}} \,
 e_{\bar{\mathfrak{p}}}(\bar{\chi})
 P_{\frak{p}}({\bar\chi}, u^{-1}) P_{\bar{\frak{p}}} ({ \chi},
 u^{-1})\,\,u^{-\mathfrak{f_{\bar{\mathfrak{p}}}}(\bar{\chi})}
 \]
 for all Artin-character $\chi$ of $G.$
\end{cor}

\begin{proof} This measure is just a twist of the measure
$\mathcal{L}'_{\bar\psi}$ by a unit in $\Lambda_D(G)$. This follows
from the functional equation of the $L$ function involved. We defer
the proof of this corollary for the next section (compare with
corollary \ref{p-LCM1} and corollary \ref{p-LCM2}).
\end{proof}

Now we are ready to state the two-variable Main Conjecture
\cite[page 37, Theorem 4.1]{rubin} and \cite[page 413, Theorem
1]{yager}:

\begin{thm}[Rubin,Yager]
As ideals in $\La_D(G)$ there is the following equality
\[\mathrm{char}\left(\mathcal{X}(K_\infty)\otimes_\zp T_\pi^*\right)=(\mathcal{L}_{\bar{\psi}}),\]
where the left hand side denotes the characteristic ideal associated
with the $\La(G)$-torsion module $\mathcal{X}(K_\infty)\otimes_\zp
T_\pi^*$ by the structure theory from commutative algebra.
\end{thm}

\begin{lem}
The module $\mathcal{X}(K_\infty)\otimes_\zp T_\pi^*$ is
$S$-torsion, in particular $\mathcal{X}(K_\infty)\otimes_\zp
T_\pi^*\in \mathfrak{M}_H(G),$ and thus $\mathcal{L}_{\bar{\psi}}\in
(\La(G)_S)^\times\subseteq (\La(G)_{S^*})^\times.$
\end{lem}

\begin{proof}
By Proposition \ref{Selmerind} $\mathcal{X}(K_\infty)\otimes_\zp
T_\pi^*$ being $S$-torsion is equivalent to $X(E/K_\infty)$ being
$S$-torsion, which is just remark \ref{torConjThm}. The rest of the
claim follows from the two-variable Main Conjecture above and the
commutative diagram
\[\xymatrix{
  0 \ar[r]& K_1(\La(G)) \ar@{=}[d]              \ar[r]^{ } & K_1(\La(G)_S)  \ar@{^(->}[d]  \ar[r]^{ } & K_0(S\mathrm{-tor}) \ar@{^(->}[d]  \ar[r]^{ } &
  0
   \\
  0 \ar[r]^{ } & K_1(\La(G)) \ar[r]^{ } & K_1(Q(G)) \ar[r]^{ } & K_0(\mathrm{tors}) \ar[r]^{ } & 0   },\]
  where $Q(G)$ denotes the maximal ring of quotients of $\La(G),$ while $\mathrm{tors} $ denotes
  the category of all finitely generated $\La(G)$-torsion modules.
\end{proof}

Note that, although the Ore-sets $S^*_G$ and $S^*_\G$ are different
in general, we have a natural isomorphism
$\La(\G)_{S^*_\G}\cong\La(\G)_{S^*_G}$ by \cite[Lemma 4.6]{sch-ven2}
where in the second case the localization is formed with respect to
the $G$-module structure of $\La(\G).$ Hence we have a commutative
base change square
\[\xymatrix{
  {\La(G)} \ar@{^(->}[d]^{\iota} \ar@{^(->}[r] &  {\La(G)_S}\ar@{^(->}[d]^{\iota_S}\\
  {\La(\G)} \ar@{^(->}[r] & {\La(\G)_S} ,  }\]

which  induces the following diagram with exact rows by
functoriality of $K$-theory

  \[\xymatrix{
    0   \ar[r]^{ } & K_1(\La(G)) \ar[d]_{\iota_*} \ar[r]^{ } & K_1(\La(G)_{S^*}) \ar[d]_{{(\iota_S)_*}} \ar[r]^{\partial} & K_0(\mathfrak{M}_H(G)) \ar[d]_{\La(\G)\otimes_{\La(G)}-} \ar[r]^{ } & 0   \\
     & K_1(\La(\G)) \ar[r]^{ } & K_1(\La(\G)_{S^*}) \ar[r]^{ } & K_0(\mathfrak{M}_{\mathcal{H}}(\G)) \ar[r]^{ } & 0.   }\]

It follows that $\mathcal{L}:=(\iota_S)_*(\mathcal{L}_{\bar{\psi}})$
is automatically a characteristic element for $[X(E/K_\infty)]$ by
Proposition \ref{Selmerind}. A naive hope would be that this is the
desired $p$-adic $L$-function for $E.$ In order to check this we
need the next lemma which describes the evaluation of $\mathcal{L}$
at representations.

\begin{lem}\label{extended measure}
For all Artin representations $\rho$ of $\G$  one has
\[\mathcal{L}(\rho)=\mathcal{L}_{\bar{\psi}}(Res^\G_G\rho),\]
where $Res^\G_G\rho$ denotes the restriction to the subgroup $G.$
With other words $\mathcal{L}$ is induced by the measure
$\mu_{\mathcal{L}_{\bar{\psi}}}$ trivially extended
\[\int_\G fd\mu_\mathcal{L}=\int_G f_{|G} d\mu_{\mathcal{L}_{\bar{\psi}}},\]
i.e.\ from the image of $\mu_{\mathcal{L}_{\bar{\psi}}}$ under the
natural map $\La(G)\hookrightarrow \La(\G).$
\end{lem}

\begin{proof}
Upon comparing with the definition of evaluation in \cite{cfksv} the
statement follows from the following diagram, which is obviously
commutative:
\[\xymatrix{
  {\La(G)_{S^*}} \ar[d]_{ } \ar[rr]^{(Res^\G_G\rho)\otimes pr_{|G}\phantom{mmmm}} && M_n(\O)\otimes_\zp\La(\Gamma)_{S^*} \ar@{=}[d]^{ } \\
   {\La(\G)_{S^*}} \ar[rr]^{\rho\otimes pr\phantom{mmmm}} && M_n(\O)\otimes_\zp\La(\Gamma)_{S^*} ,  }\]
   where $ {pr}:\La(\G)_{S^*}\twoheadrightarrow\La(\Gamma)_{S^*}$ denotes the canonical projection.
  \end{proof}

 The next task is to understand the (irreducible) Artin representations of $\G.$ By \cite[page 62]{serre}
 and the fact that $\G$ is a semi-direct product \[\G\cong G\rtimes\Delta\]
 one immediately obtains that an irreducible such representation is either of

 {\bf Typ A :} $\rho$ is one dimensional,


 or

 {\bf Typ B:} $\rho=\Ind^G_\G\chi$ is of dimension two where $\chi$ is a one-dimensional character of
 $G$ with $\chi \neq \chi^c$.

Indeed to see that the only irreducible representations of dimension
bigger than one are those of dimension two and they are of the form
described above one has simply to use Frobenious reciprocity. For if
an Artin representation $\rho$ of $\G$ has dimension $n$ with $n
\geq 2$ and we consider its restriction to $G$ then $Res^G_\G\rho =
\oplus_{i=1}^n \chi_i$ for $n$ one dimensional representations of
$G$. But then for every $\chi_j$, with $1 \leq j \leq n$, by
applying Frobenious Reciprocity we have
\[
0 \neq (Res_\G^G\rho,\chi_j)_G = (\rho, Ind_{\G}^G\chi_j)_\G
\]
and the representation $Ind_{\G}^G\chi_j$ is of dimension two. That
is, every representation of dimension greater (or equal) than two
has a subrepresentation of dimension two. The irreducibility
assumption allows us to conclude our claim.

From Lemma \ref{extended measure} it is now clear that
\[\mathcal{L}(\epsilon)=\mathcal{L}_{\bar{\psi}}(Res^G_\G\epsilon)=\mathcal{L}(\mathbf{1})\]
where $\mathbf{1}$ is the trivial representation. But this is   not
compatible with the interpolation property in Conjecture \ref{Conj2}
because the periods $\Omega_+$ and $\Omega_-$ are interchanged. This
is clear from an philosophical point of view as the periods
$\Omega_\pm$ arise from paths $\gamma_\pm\in
\H_1(E(\Co),\Z)^\pm\subseteq \H_1(E(\Co),\Zp)^\pm\cong T_pE,$ while
the period $\Omega_\infty$ corresponds to a choice of $\gamma\in
T_\pi E$ (via the same identification). Thus we define the following
correction term which describes the change of complex periods

\[\mathcal{L}_\Omega:= \frac{\Omega_+}{\Omega_\infty} \frac{1+c}{2} + \frac{\Omega_-}{\Omega_\infty}\frac{1-c}{2}\]

\begin{lem}
$\mathcal{L}_\Omega\in \La_\zp(\G)^\times.$
\end{lem}

\begin{proof}In order to prove this lemma we need to understand the
relation between the Neron periods $\Omega_+$, $\Omega_-$ and the
period $\Omega_\infty$, and in particular show that the ratios
$\frac{\Omega_+}{\Omega_\infty}$ and
$\frac{\Omega_-}{\Omega_\infty}$ belong to $\zp^\times$.  Then we
have that $\mathcal{L}_\Omega\in \La_\zp(\G)$ and it is easily seen
then that the element $\frac{\Omega_\infty}{\Omega_+} \frac{1+c}{2}
+ \frac{\Omega_\infty}{\Omega_-}\frac{1-c}{2} \in \La_\zp(\G)$ is
its inverse.  Recall that we write $\Lambda$ for the lattice
associated to $E$ and we have defined $\Omega_\infty$ by $\Lambda =
\Omega_\infty \mathfrak{f}:\psi$. For our aims we may actually
assume that $\Lambda = \Omega_\infty O_K$ as $\mathfrak{f}=(f)$ for
some $f \in K \hookrightarrow K_{\p}\cong \zp$ which is a $p$-adic
unit since $(\mathfrak{f},\mathfrak{p})=1$. Recall that we view
$O_K=\Z + \Z \sqrt{-d_K}$ as a lattice in $\mathbb{C}$ by our fixed
embedding $K \hookrightarrow \mathbb{C}$.

It is a well-known fact that the Neron periods have the property
that $\tau:=\frac{\Omega_-}{\Omega_+}$ is a totally imaginary number
and moreover the lattice $\Lambda' := \mathbb{Z} \Omega_+ +
\mathbb{Z} \Omega_-$ is isomorphic to $\Lambda$ as
$\mathbb{Z}[\frac{1}{2}]$-lattices. As $p \neq 2$ we may work for
our purposes with the lattice $\Lambda'$ which we keep denoting by
$\Lambda$. As the lattice $\Lambda$ has $CM$ by $O_K$ we have that
$\frac{1}{\Omega_+}\Lambda \subseteq K$ is a fractional ideal of $K$
\cite[page 164]{silv1}. In particular we have that $\tau \in K$ and
that $\Z + \Z \tau = O_{K}\alpha$ for some $\alpha \in K$. We write
$\tau = s \sqrt{-d_K}$ and $\alpha = \alpha_{1} + \alpha_{2}
\sqrt{-d_K}$ with $s,\alpha_1,\alpha_2 \in \Q$. We will show that $s
\in \mathbb{Z}_{(p)}^{\times}$ and $\alpha \in
O_{\mathfrak{p}}^{\times}$. Note that this is enough for our
purposes because we have
\[
\mathbb{Z} +
\mathbb{Z}\tau=\mathbb{Z}+\mathbb{Z}\frac{\Omega_{-}}{\Omega_{+}}=\frac{\Lambda}{\Omega_{+}}=\frac{\Omega_{\infty}}{\Omega_{+}}O_K
\]

The lattice $\Z + \Z \tau$ has $CM$ by $O_{K}$. That means that we
have that
\[
O_{K}=\Z + \Z \sqrt{-d_K} = \{a+b \tau| a,b \in \Z, 2br \in \Z\,\,
and\,\,b(r^{2}+d_K s^{2}) \in \Z \}
\]
As $\sqrt{-d_K}$ should belong to the set on the right we have that
$bs =1$ and $d_K s \in \Z$. That is, $s=\frac{s'}{d_K}$ for some $s'
\in \Z$. As $b=\frac{1}{s}=\frac{d_K}{s'} \in \Z$ we conclude that
$s' | d_K$. Hence $s=(\frac{d_K}{s'})^{-1} \in \Z_{p}^{\times}$ as
$D \in \Z_{p}^{\times}$. That is, $\tau \in
O_{\mathfrak{p}}^{\times}$.

Now we show that $\alpha \in O_{\mathfrak{p}}^{\times}$. We have $\Z
+ \Z \tau = O_{K}\alpha$. We take the completion at $\mathfrak{p}$.
It is enough to show that $(\Z + \Z \tau) \otimes_{O_{K}}
O_{\mathfrak{p}} = O_{\mathfrak{p}}$. This follows trivially from
the fact that $\tau \in O_{\mathfrak{p}}^{\times}$. Hence we obtain
that also $\alpha$ is a unit.
\end{proof}

Finally we set
\[\mathcal{L}'_E:=\mathcal{L}\mathcal{L}_\Omega^{-1}\in
K_1(\La_D(\G)_{S^*}).\]

\begin{prop}
$\mathcal{L}'_E$ satisfies the interpolation property in Conjecture
\ref{Conj2}.
\end{prop}

\begin{proof}
 We calculate the fraction $\frac{\mathcal{L}(\check{\rho} )}{\mathcal{L}'_E(\check{\rho})},$  where the denominator here is   the right hand side
 of the desired
 interpolation formula in Conjecture \ref{Conj2}, in a purely formal way in order to compare all
 terms showing up, in particular the possible presence of zeroes does not matter at all.

  We start
 with representations $\rho=\Ind^G_\G\chi$ of type B. Then $\check{\rho}_{|_G}=\bar\chi\oplus\bar \chi^c$
   where $\chi^c(g):=\chi(cgc^{-1})$  with complex conjugation $c\in \mathcal{G}.$

 $$
 \frac{\mathcal{L}(\check{\rho})}{\mathcal{L}'_E(\check{\rho})}=
 \frac{\mathcal{L}_{\bar\psi}(\bar \chi)\cdot\mathcal{L}_{\bar\psi}(\bar \chi^c)}{\mathcal{L}'_E(\check{\rho})}=
 \frac{\displaystyle{ \frac{1}{\Omega_\infty^2}}} {\Omega_+^{-1}\Omega_-^{-1}} \
 \frac{u^{-( \mathfrak{f}_{\bar{\mathfrak{p}}}(\bar{\chi})+\mathfrak{f}_{\bar{\mathfrak{p}}}(\bar{\chi}^c))}e_{\bar{\mathfrak{p}}}(\bar \chi)e_{\bar{\mathfrak{p}}}(\bar \chi^c)}{u^{-\mathfrak{f}_p(\check{\rho})}\,e_p(\check{\rho})}
 $$

 $$
 \cdot\;\;\frac{P_\mathfrak{p}(\bar \chi,u^{-1})
       P_{\bar{ \mathfrak{p}}}(  \chi,u^{-1})
       P_\mathfrak{p}(\bar \chi^c,u^{-1})
       P_{\bar{\mathfrak{p}}}(\chi^c,u^{-1})}
       {P_p(\check{\rho},u^{-1})
       P_p(\rho,w^{-1})^{-1}
       P_p(E,\rho,\frac{1}{p})}
      \frac{L(\bar{\psi} \chi,1)\, L(\bar{\psi}\chi^c,1)}{L(E,\rho,1)}
  $$

  $$
  = \frac{  \Omega_+ \Omega_-}{\:\:\:\:\Omega_\infty\  \ \Omega_\infty} \
     \frac{u^{-( \mathfrak{f}_{\mathfrak{p}}(\bar{\chi})+\mathfrak{f}_{\bar{\mathfrak{p}}}(\bar{\chi}))} e_{\mathfrak{p}}(\bar{\chi})e_{\bar{\mathfrak{p}}}(\bar{\chi})}
     {u^{-\mathfrak{f}_{ p}(\check{\rho})} \ e_p(\check{\rho})}  \
 $$
 $$
   \cdot\;\frac{P_\mathfrak{p}(\bar\chi,u^{-1})
         P_{\bar{\mathfrak{p}}}(\bar\chi,u^{-1})
         P_\mathfrak{p}( \chi^c,u^{-1})
         P_{\bar{\mathfrak{p}}}( \chi^c,u^{-1})\phantom{mm}P_{p}(\rho,w^{-1})\phantom{mm}}
        {\phantom{mmm}P_{p}(\check{\rho},u^{-1})\phantom{mmii}
         P_{\bar{ \mathfrak{p}}}(\bar{\psi} \chi,\frac{1}{p})\phantom{m}
         P_{\bar{\mathfrak{p}}}(\bar{\psi}\chi^c,\frac{1}{p})\phantom{i}P_\mathfrak{p}(\bar{\psi}\chi^c,\frac{1}{p})
         P_\mathfrak{p}(\bar{\psi} \chi,\frac{1}{p})}
 $$
 $$
  \cdot\;\;\frac{L(\bar{\psi} \chi,1)\cdot L(\bar{\psi} \chi^c,1)}{L((\textrm{Ind}\bar{\psi})\otimes \textrm{Ind}\chi,1)}
   =  \alpha^2\tau
 $$

 with $\phantom{mmmm}\tau:=\displaystyle{\frac{\Omega_-}{\Omega_+}},\ \;\;
        \alpha:=\frac{\Omega_+}{\!\!\!\Omega_\infty}\in\mathcal{O}_\mathfrak{p}^\times\cap K$.\\

Let us comment on the calculations above. We start with our
considerations on the $\epsilon$-factors. Note that as the
representation $\check{\rho}$ is induced from $\bar{\chi}$ we have
for the conductor of $\check{\rho}$, $N_{\check{\rho}}=
DN_{K/\Q}(N_{\bar{\chi}})$ where $N_{\bar{\chi}}$ the conductor of
$\bar{\chi}$. Locally at $p$ that means
$\mathfrak{f}_p(\check{\rho})=
\mathfrak{f}_{\mathfrak{p}}(\bar{\chi}) +
\mathfrak{f}_{\bar{\mathfrak{p}}}(\bar{\chi})$ as $p$ splits in $K$.
Concerning the epsilon factors note that we know that they are
inductive in degree zero. In particular we have
\[
e_p(Ind_{G}^{\G}(\bar{\chi} \ominus 1))= e_{\mathfrak{p}}(\bar{\chi}
\ominus 1)e_{\bar{\mathfrak{p}}}(\bar{\chi} \ominus 1)=
\frac{e_{\mathfrak{p}}(\bar{\chi})e_{\bar{\mathfrak{p}}}(\bar{\chi})}{e_{\mathfrak{p}}(1)e_{\bar{\mathfrak{p}}}(1)}
\]
But also
\[
e_p(Ind_{G}^{\G}(\bar{\chi} \ominus
1)=\frac{e_p(Ind_{G}^{\G}(\bar{\chi}))}{e_p(1 \oplus
\epsilon)}=\frac{e_p(\check{\rho})}{e_p(1)e_p(\epsilon)}
\]
where as always $\epsilon$ is the non-trivial character of
$Gal(K/\Q)$. Hence we obtain
\[
e_p(\check{\rho})=\frac{e_{\mathfrak{p}}(\bar{\chi})e_{\bar{\mathfrak{p}}}(\bar{\chi})}{e_{\mathfrak{p}}(1)e_{\bar{\mathfrak{p}}}(1)}e_p(1)e_p(\epsilon)
\]
But note that $p$ split in $K$ we have
$e_p(1)=e_p(\epsilon)=e_{\mathfrak{p}}(1)=e_{\bar{\mathfrak{p}}}(1)=1$.
Hence $e_p(\check{\rho}) =
e_{\mathfrak{p}}(\bar{\chi})e_{\bar{\mathfrak{p}}}(\bar{\chi})$. Now
we explain the ratio of the Euler factors. The starting point is the
inductive properties of the Euler factors. In particular as $p$
splits in $K$ we have that $P_p(\rho,X) =
P_\mathfrak{p}(\chi,X)P_{\bar{\mathfrak{p}}}(\chi,X)$. Using this
and the relations
\[u^{-1}=\frac{\psi(\frak{p})}{p}=\frac{\bar{\psi}(\bar{\frak{p}})}{p},\;\;  w^{-1}=\frac{\psi(\bar{\frak{p}})}{p}=\frac{\bar{\psi}( \frak{p})}{p}.\]
we see that the Euler factors cancel (the Euler factors which are
opposite to each other with respect to the fraction line cancel each
other). Finally for the last equation we use the inductiveness of
$L$-functions and the following well-known fact from representation
theory implied by Frobenius reciprocity
 $$
 (\textrm{Ind}\bar{\psi})\otimes \textrm{Ind} \chi= \textrm{Ind}(\bar{\psi}\otimes \textrm{Res\ Ind}\chi)=\textrm{Ind}(\bar{\psi}\otimes(\chi\oplus
 \chi^c)).
 $$

Note that, since $\bar{\psi}=\psi^c$ by \cite[page 559]{katz} also
$E=\Ind\bar{\psi}$ holds.

Now we turn to those $\rho$ of type A, i.e $\rho$ is now one
dimensional. We have

$$
  \frac{\mathcal{L}(\check{\rho})}{\mathcal{L}'_E(\check{\rho})}=
  \frac{\mathcal{L}_{\bar\psi}(\bar
  \rho_{|G})}{\mathcal{L}'_E(\check{\rho})}=
$$

$$
  = \frac{\displaystyle{ \frac{1}{\Omega_\infty}}} {\Omega_{\rho(c)}^{-1}} \
 \frac{u^{-\mathfrak{f}_{\bar{\mathfrak{p}}}(\bar{\rho}_{|G})}e_{\bar{\mathfrak{p}}}(\bar{\rho}_{|G})}{u^{-\mathfrak{f}_p(\bar{\rho})}\,e_p(\bar{\rho})}\frac{P_\mathfrak{p}(\bar{\rho}_{|G},u^{-1})
         P_{\bar{\mathfrak{p}}}(\bar\rho_{|G},u^{-1})
         P_{p}(\rho,w^{-1})}
        {P_{p}(\bar{\rho},u^{-1})P_p(E,\rho,p^{-1})}
 \frac{L(\bar{\psi} \rho_{|G},1)}{L(E,\rho,1)}=
 $$
 $$
={ \frac{\Omega_{\rho(c)}}{\Omega_\infty}}\
 \frac{u^{-\mathfrak{f}_{\bar{\mathfrak{p}}}(\bar{\rho}_{|G})}e_{\bar{\mathfrak{p}}}(\bar{\rho}_{|G})}{u^{-\mathfrak{f}_p(\bar{\rho})}\,e_p(\bar{\rho})}\frac{P_\mathfrak{p}(\bar{\rho}_{|G},u^{-1})
         P_{\bar{\mathfrak{p}}}(\bar\rho_{|G},u^{-1})
         P_{\mathfrak{p}}(\rho_{|G},w^{-1})}
        {P_{\mathfrak{p}}(\bar{\rho}_{|G},u^{-1})P_{\mathfrak{p}}(\bar{\psi}\rho_{|G},p^{-1})P_{\bar{\mathfrak{p}}}(\bar{\psi}\rho_{|G},p^{-1})}
 \frac{L(\bar{\psi} \rho_{|G},1)}{L(E,\rho,1)}=
 $$

 $$
  \frac{\Omega_{\rho(c)}}{\Omega_\infty}\frac{L(\bar{\psi}\rho_{|G},1)}{L(\textrm{Ind}(\bar{\psi}\otimes\rho_{|G}),1)}
  $$
  $$
  =  \left\{\begin{array}{lll}
              \alpha &             & \rho(c)=+1\\
                     & \textrm{if} & \\
          \alpha\tau &             & \rho(c)=-1
          \end{array}\right.
  $$

 \vspace{1cm}
  Recall that $\tau=\displaystyle{\frac{\Omega_-}{\Omega_+}},\
        \alpha=\frac{\Omega_+}{\!\!\!\Omega_\infty},\
        \alpha\tau=\frac{\Omega_-}{\!\!\!\Omega_\infty}$ all belong to $\mathcal{O}_\mathfrak{p}^\times\cap K$.
Let us comment on the above computations. We start with the epsilon
factors. First note that as the extension $K$ is unramified at $p$
we have that $\mathfrak{f}_p(\bar{\rho}) =
\mathfrak{f}_{\bar{\mathfrak{p}}}(\bar{\rho}_{|G})$ since as adelic
characters $\bar{\rho}_{|G}= \bar{\rho} \circ N_{K/\Q}$ and $p$
splits in $K$, that is locally when we identify
$K_{\bar{\mathfrak{p}}}$ with $\Q_p$ the two characters are equal.
This explains also the equality
$e_p(\bar{\rho})=e_{\bar{\mathfrak{p}}}(\bar{\rho}_{|G})$. Note that
also this remark explains the equality of Euler factors
$P_{p}(\rho,X) = P_{\mathfrak{p}}(\rho_{|G},X)$ and similarly for
$\bar{\rho}$. The rest is just trivial inspection of the formula.
Finally for the $L$-functions we use as before Frobenius reciprocity
and inductive properties of $L$-functions.

Since it is easily checked that
  $$
  \mathcal{L}_{\Omega}(\check{\rho}) =
  \left\{\begin{array}{lllll}
    \alpha^2\tau & & \rho&=&\textrm{Ind}\chi\\
    \alpha & \in D^\times,\ \textrm{if} & \rho(c)&=&+1\ \\
    \alpha\tau & & \rho(c)&=&-1
   \end{array}\right.
   $$

   for irreducible $\rho,$ we obtain that
   \[\mathcal{L}'_E(\check{\rho})\mathcal{L}_{\Omega}(\check{\rho})=\mathcal{L}(\check{\rho})\]
   for all Artin representations $\rho$ as desired. 
\end{proof}

Now let $K(F_{\infty})$ be the maximal abelian extension of
$\mathbb{Q}$ inside $ F_\infty$   in which $ p$   does not ramify
and $L = K (F_{\infty})_{\frak{P}}$ its completion at some
$\frak{P}$ lying over $p$. Then $\mathcal{L}'_E$ can be replaced by
an element $\mathcal{L}_E$ defined already over $\mathcal{O}$, the
ring of integers of $L$.

\begin{thm}
Assuming Conjecture \ref{torConj} there exists $\mathcal{L}_E\in
K_1(\La_\mathcal{O}(\G)_S)$ satisfying Conjecture \ref{Conj2} and
\ref{Conj3}.
\end{thm}

\begin{proof}
This follows from the explanation before Lemma \ref{extended
measure} and Theorem \ref{L-descent} in the appendix. Indeed
assumption (ii) holds because $\mathcal{O}(\rho)$ is contained in
$\mathcal{O}(\mu(p))$ as the values of $\chi$ (and thus $\rho=
\textrm{Ind}\chi$ or $\chi$)are in $\mu_{(p-1)p^\infty};$ (iii) is
clear from the construction, (iv)  follows from the Lemma below
while (v) again follows from the explanation before Lemma
\ref{extended measure}.
\end{proof}

\begin{lem}[Deligne-Conjecture, Blasius] In the CM setting the following algebraicity
result holds:
\[
\frac{L_{\{p\}}(E, \rho, 1)}{\Omega_+^{d^+ (\rho)}
\Omega_-^{d^-(\rho)}}\, e_p(\check{\rho})
\frac{P_p(\mathop{\rho}\limits^{\vee}, u^{-1})} {P_p(\rho, w^{-1})}
u^{-{\frak{f}}_p (\check{\rho})} \in L(\rho).
\]
Moreover these values are $p$-adically integral.
\end{lem}
\begin{proof}That the values are $p$-integral follows from the fact
that are obtained as values of a $p$-integral valued measure. So we
need to prove that actually the values are in $L(\rho)$ and of
course it is enough to prove it for $\rho$ an irreducible Artin
representation. Note that as $p$ splits we can identify $\Q_p =
K_{\mathfrak{p}}$. Moreover we have that $u,w \in \Q_p$ hence we
need simply to prove that
\[
\frac{L(E, \rho, 1)}{\Omega_+^{d^+ (\rho)} \Omega_-^{d^-(\rho)}}\,
e_p(\check{\rho}) \in L(\rho)
\]
If $\rho$ is one dimensional then this is well known. Hence we
consider the case where $\rho = Ind_{\mathcal{G}}^G\chi$. Then the
above statement is equivalent to
\[
\frac{L(E/K, \chi, 1)}{\Omega_\infty^{2}}\, (e_{\mathfrak{p}}(\chi)
e_{\bar{\mathfrak{p}}}(\chi))^{-1} \in L(\chi)
\]
using the fact the $\Omega_+ = \Omega_- = \Omega_{\infty}$ up to
elements in $K^\times$, $e_{\mathfrak{p}}(\bar{\chi})
e_{\bar{\mathfrak{p}}}(\bar{\chi})=e_p(\check{\rho})$ and the
duality of the epsilon factors. Moreover we have $L(E/K, \chi,
1)=L(\psi,\chi,1)L(\bar{\psi},\chi,1)$. From Blasius proof of
Deligne's conjecture for Hecke characters (of CM fields)
\cite{Blasius} we know that
\[
\frac{L(\psi,\chi,1)}{\Omega_{\infty}}c(\psi,\chi)^{-1} \in
K(\chi),\,\,\,\frac{L(\bar{\psi},\chi,1)}{\Omega_{\infty}}c(\bar{\psi},\chi)^{-1}
\in K(\chi)
\]
where $c(\psi,\chi)$ and $c(\bar{\psi},\chi)$ are defined as in
\cite{Blasius2} (see page 65 for the definition and page 67 for the
factorization of Deligne's periods for the Hecke characters
$\psi\chi$ and $\bar{\psi}\chi$). Moreover we have that
$c(\psi,\chi)c(\bar{\psi},\chi)=e_{\mathfrak{p}}(\chi)
e_{\bar{\mathfrak{p}}}(\chi)$ up to elements in $L(\chi)^\times$.
This can be seen from the definition of the periods $c(\psi,\chi)$
and $c(\bar{\psi},\chi)$. From \cite[page 65]{Blasius2} we have that
the element $c(\psi,\chi)$ is an element in $K(\chi)
\otimes_{\mathbb{Q}} \bar{\mathbb{Q}}$ characterized from the
reciprocity law
\[
(1 \otimes \tau) c(\psi,\chi) = (\chi \circ Ver_{\psi})(\tau)
c(\psi,\chi)
\]
where $Ver_{\psi}:G_{\Q} \rightarrow G^{ab}_{K}$ is the
half-transfer map of Tate associated to the CM-type of the character
$\psi$. As it is explained by Blasius this reciprocity law
characterizes the element $c(\psi,\chi)$ up to elements in
$K(\chi)^{\times}$. From the properties of the half-transfer map of
Tate one has that
\[
Ver_{\psi}(\tau) \cdot Ver_{\bar{\psi}}(\tau) = Ver(\tau), \,\,\tau
\in G_{\Q}
\]
where $Ver:G_{\Q} \rightarrow G^{ab}_{K}$ is the classical transfer
map. But then as is shown in \cite[page 70]{Blasius2}, the product
of the all epsilon factors $\prod_{\mathfrak{q}}
e_{\mathfrak{q}}(\chi)$ has exactly the same reciprocity law as the
product $c(\psi,\chi)c(\bar{\psi},\chi)$ with respect the operation
of $G_\Q$ and hence they are equal up to elements in
$K(\chi)^\times$. The characters $\chi$ that we consider are only
ramified at $p$ or at the primes $\mathfrak{q}$ that divide
$\mathfrak{f}_\psi$. But for the epsilon factors
$e_{\mathfrak{q}}(\chi)$ for $\mathfrak{q} | \mathfrak{f}_{\psi}$ we
have that $e_{\mathfrak{q}}(\chi)=e_{q}(\rho)$ up to elements in
$K(\chi)^{\times}$ where $q:=\mathfrak{q} \cap \mathbb{Q}$. But
moreover it is well known \cite[page 330]{Deligne} that $e_q(\rho) =
e_q(det(\rho))$ up to element in $\mathbb{Q}(\rho)^{\times}$. But
$e_q(det(\rho))$ is a Gauss sum hence an element in
$L(\rho)^{\times}$ see \cite[pages 103 and 104]{schappacher}. Hence
we conclude that also $e_{\mathfrak{q}}(\chi) \in L(\chi)^{\times}$.
Hence putting all the above considerations together we conclude that
$e_{\p}(\chi)e_{\bar{\p}}(\chi)$ is equal up to elements in
$L(\chi)^{\times}$ to $c(\psi,\chi)c(\bar{\psi},\chi)$.
\end{proof}

We finish this section with a last comment. It seems natural to ask
if the measure $\mathcal{L}_{\bar{\psi}}$ has already values in
$\mathbb{Z}_p$ or could be suitably modified by a unit in the
Iwasawa algebra $\Lambda_{D}(G)$ to still interpolate the critical
values and take values in $\mathbb{Z}_p$. Let us first try to
explain why we believe that this measure does not take values in
$\mathbb{Z}_p$. We recall its interpolation property:
\[
 \mathcal{L}_{\bar\psi} { (\bar\chi) } = \mathop{\int}\limits_{G}
 { \bar\chi } d\mu = \frac{L ({\bar\psi} \chi, 1)}{{\Omega_\infty}} \,
 e_{\mathfrak{p}}(\bar{\chi})
 P_{\frak{p}}({\bar\chi}, u^{-1}) P_{\bar{\frak{p}}} ({ \chi},
 u^{-1})\,\,u^{-\mathfrak{f_{\mathfrak{p}}}(\bar{\chi})}
 \]
 for all Artin-character $\chi$ of $G.$ Now we note that in order to
 prove that this measure lies in $\Lambda_{\mathbb{Z}_p}(G)$ is
 equivalent to show the following rationality result
 \[
 \frac{L ({\bar\psi} \chi, 1)}{{\Omega_\infty}} \,
 e_{\mathfrak{p}}(\chi)^{-1} \in K(\chi)
 \]
since we are in the ordinary case i.e.
$K_{\mathfrak{p}}=\mathbb{Q}_p$. As we explained in the proof above
we know from Blasius work that
\[
\frac{L(\bar{\psi}\chi,1)}{\Omega_{\infty}}c(\bar{\psi},\chi)^{-1}
\in K(\chi)
\]
hence one needs to understand if
$e_{\mathfrak{p}}(\chi)=c(\bar{\psi},\chi)$ up to elements in
$K(\chi)^\times$ for finite characters $\chi$ of $G$. However this
cannot be the case for all finite characters of $Gal(K^{ab}/K)$.
(However we remark that it is the case when $\chi$ is cyclotomic
i.e. $\chi^c=\chi$.) Indeed from Blasius \cite{Blasius2} (page 66)
we have that the extension of $K$ defined by adjoining to $K$ the
values $c(\bar{\psi},\chi))$ for all finite order characters of
$Gal(K^{ab}/K)$ is an abelian extension of $K$ not included in
$\mathbb{Q}^{ab}$. However the epsilon factors are just Gauss sums
hence they can generate only extensions in $\mathbb{Q}^{ab}$. In
particular we have that the two ``periods'' cannot be equal up to
elements in $K(\chi)^\times \subseteq \mathbb{Q}^{ab}$ (see also the
comment in \cite{schappacher} page 109). Concerning the other
question one may speculate that there is a measure
$\mathcal{L}^{?}_{\bar{\psi}}$ of $G$ such that
\[
\mathcal{L}^{?}_{\bar\psi} { (\bar\chi) } = \mathop{\int}\limits_{G}
 { \bar\chi } d\mu = \frac{L ({\bar\psi} \chi, 1)}{c_p(\bar{\psi},\chi){\Omega_\infty}} \,
 P_{\frak{p}}({\bar\chi}, u^{-1}) P_{\bar{\frak{p}}} ({ \chi},
 u^{-1})\,\,u^{-\mathfrak{f_{\mathfrak{p}}}(\bar{\chi})}
\]
for some canonically normalized element $c_p(\bar\psi,\chi)$ equal
to $c(\bar\psi,\chi)$ up to elements in $K(\chi)^\times$. This
measure would have values in $\mathbb{Z}_p$.

\section{CM-modular forms}

In this section we would like to indicate how most of our results
can be extended to the case of CM-modular forms. Our reference for
the theory of CM-modular forms is Ribet's aricle \cite{ribet} as
well as the last section of Schappacher's book \cite{schappacher}.

We start by fixing our setting. Let $f$ be cuspidal newform of
weight $k \geq 2$, Nebentypus $\epsilon_{f}$ and level $N$. We pick
a number field $F$ that contains all $a_n$ of $f=\sum_{n \geq 1}a_n
q^n$. We want now to formulate a $GL_2$-Main Conjecture for $f$. Our
starting point is the following theorem.

\begin{thm}[Eichler-Shimura-Deligne-Scholl-Jannsen] There exists a motive
$M(f)$ defined over $\Q$ with coefficients in $F$ of rank two over
$F$ such that
\[
L^{*}_{N}(M(f),s):=(\prod_{(p,N)=1 }det_{F}(1-Fr_p\,p^{-s}|
H_{\ell}(M(f))^{-1})_{\tau}=(\sum_{(n,N)=1}a_n^{\tau}n^{-s})_{\tau}
\]
where $\tau \in Hom(F,\mathbb{C})$, $\ell \neq p$ prime, $Fr_p$ is a
geometric Frobenius element at $p$ and $Re(s)\gg 0$
\end{thm}
\begin{proof} We refer to the book of Schappacher \cite{schappacher} in page 139 and
the references there.
\end{proof}

It is known \cite[page 240, 14.10]{kato} that the Kummer dual
$M(f)^*(1)$ of $M(f)$ is isomorphic to $M(f^*)(k),$ where
$f^*=\sum_{n \geq 1}\bar{a_n} q^n$ is the dual cusp form.

We fix  an embedding $\tau : F \hookrightarrow \bar{\mathbb{Q}},$
we let $p\geq 5$ be a rational prime and let $\lambda$ be the prime
of $F$ above $p$ corresponding to our fixed embedding $ F
\hookrightarrow \bar{\q}\hookrightarrow\bar{\qp}.$ We write
$\rho_{\lambda}:=\rho_{f,\lambda}:G_{\Q} \rightarrow
GL_{2}(F_{\lambda})$ for the  associated $G_\q$- representation
given by the $\lambda$-adic realisation $V:=V(f):=V_{F_\lambda}(f)$
of $M(f)$. Moreover we assume that $p$ is a {\it good ordinary}
prime for $M(f)$, this is equivalent to $p$ being relative prime to
$N$ and $a_p$ a $\lambda$-adic unit. Note that with $f$ also $f^*$
is good ordinary at $p,$ see \cite[prop.\ 17.1]{kato}.

 We consider the $p$-adic Lie
extension $F_\infty$ of $\mathbb{Q}$ determined by the image of
$\rho_{\lambda}$, i.e. $\mathcal{G}:=Gal(F_{\infty}/\Q) \cong
Im(\rho_{\lambda})$. We note that the determinant of
$\rho_{\lambda}$ is of the form \[\det(\rho_{\lambda}) \cong
\chi_{cycl}^{1-k}\epsilon_{f}^{-1}\] and hence $\mathcal{G}$
contains a closed normal subgroup $\mathcal{H}$ such that
$\mathcal{G}/\mathcal{H} \cong \zp$. In particular the setting of
Theorem 1.1 of the introduction apply to the above defined group
$\mathcal{G}$.

Next we define  the $\lambda$-primary Selmer group attached to $f$
as follows: By $V'(f)$ we denote the (unique) unramified one
dimensional $G_\qp$-subrepresentation of $V(f)$ (restricted to
$G_\qp$), which exists due to \cite[prop.\ 17.1]{kato} $f$ being
good ordinary at $\lambda.$ We fix an $\O:=\O_{F_\lambda}$-lattice $
T(f)\subseteq V(f)$ and set $T'(f):=T(f)\cap V'(f)$ and
$T''(f):=T(f)/T'(f)\subseteq V''(f):=V(f)/V'(f).$  Then, for any
integer $r,$ we define
\[Sel^{ord}(T(f)(r)/F_\infty):=\ker\big(\H^1(G_{S_p}(F_\infty),T(f)(r)\otimes \q/\z)\to \Coind^{\G_p}_\G \H^1(F_{\infty,\nu},T''(f)(r)\otimes \q/\z)\big)\]

where $G_{S_p}(F_\infty)$ denotes the Galois group of the maximal
outside $p$ unramified extension of $F_\infty$ and $\nu$ is a fixed
place of $F_\infty$ over $p.$  Using the same arguments as in the
proof of \cite[prop.\ 17.2]{kato} one easily shows that this Selmer
group coincides with the Bloch-Kato Selmer group
$Sel_{(1)}(T(f)(r)\otimes \q/\z,F_\infty)$ in
\cite[4.2.28]{fukaya-kato}. Finally, we write
  \[X:=X(T(f)(r)/F_{\infty}):=Sel^{ord}(T(f)(r)/F_{\infty})^{\vee}\]  for its
Pontryagin dual. Then the Torsion-Conjecture reads as follows

\begin{conj}[Torsion-Conjecture]\label{torConj2}
For one (and hence any) $r,$ the dual of the Selmer group is
$S^*$-torsion:
\[X(T(f^*)(r)/F_\infty)\in\mathfrak{M}_\mathcal{H}(\G).\]
\end{conj}

We let $L$ denote the same field as in the introduction in our
setting and we write $\Lambda(\G)$ for the Iwasawa algebra of $\G$
with coefficients in $D:=\mathcal{O}_L$. We denote by $\Omega_{\pm}
$
 the periods of Deligne associated to $M(f)$ with respect to our fixed    $\tau \in Hom(F,\mathbb{C}),$  that is determined up to elements in $O_{F}^\times,$ see below for more details. We
set
\[
R=\{p\}\cup\{l\neq p| \mbox{   the ramification index of $l$ in
$F_\infty/\Q$ is infinite}\}
\]
and we define $u \in \zp^\times$ by
\[
1-a_pT+p^{k-1}\epsilon_{f}(p)T^2=(1-uT)(1-wT)
\]

\begin{conj}[Existence of $p$-adic $L$-function]\label{Conj2.2}
There is a   $ \mathcal{L}_f \in K_1(\Lambda (\G)_{ S^{\normalsize
*}}) $ such that
\[
\mathcal{L}_f (\rho) = \frac{L_R(M(f), \mathop{\rho}\limits^{\vee},
1)}{\Omega_+^{d^+ (\rho)} \Omega_-^{d^-(\rho)}}\, e_p(\rho)
\frac{P_p(\rho, u^{-1})} {P_p(\mathop{\rho}\limits^{\vee}, w^{-1})}
u^{-{\frak{f}}_p (\rho)}
\]
for all Artin representations $\rho$ of $\G.$
\end{conj}

And similarly,

\begin{conj}[Main Conjecture]\label{Conj3.2}
The $p$-adic $L$-function $\mathcal{L}_f$ is a characteristic
element of $X(T(f^*)(k-1)/F_\infty):$
\[\partial \mathcal{L}_f=[X(M(f)/F_\infty)]_D.\]
\end{conj}

Actually, the last two conjectures are a consequence of the much
more general conjectures in \cite{fukaya-kato}. Indeed, theorem
4.2.22 or 4.2.26 in (loc.\ cit.) applied to the motive $M=M(f)(1)$
with coefficients in $F$ predicts the existence of $\mathcal{L}_f$
such that, for $0\leq i\leq k-2,$ the following more general
interpolation property holds (at least for $i+1\neq
\frac{k-1}{2}$)\footnote{Otherwise some Euler factors might be zero
and the formula can be rewritten by replacing $R$ by the empty set
and replacing  $\frac{P_p(\rho, u^{-1}p^i)}
{P_p(\mathop{\rho}\limits^{\vee}, up^{-i-1})}$ by
$\{P_{L,p}(W,u)P_{L,p}(\hat{W},u)^{-1}\}_{u=1}\cdot
P_{L,p}(\hat{W}^*(1),1)\cdot \prod_{l\in B} P_{L,l}(W,1),$  where
$W:=M(\rho^*)_\lambda=[\rho^*]_\lambda\otimes  M_\lambda$ and
$\hat{W}:=[\rho^*]_\lambda\otimes  \hat{V'(f)(1)}.$ }

\[
\mathcal{L}_f ( {\rho} \kappa^{-i}) = \frac{L_R(M(f),
\mathop{\rho}\limits^{\vee}, i+1)}{\Omega_+^{d^+ (\rho,i)}
\Omega_-^{d^-(\rho,i)}(2\pi \iota)^{di}}\, (i!)^d \, e_p( {\rho})
\frac{P_p(\rho, u^{-1}p^i)} {P_p(\mathop{\rho}\limits^{\vee},
up^{-i-1})} (u^{-1}p^i)^{\frak{f}_p ( {\rho})}
\]

where $d=\dim \rho$ while $d^+ (\rho,i)$ and $d^-(\rho,i)$ denote
the dimension of the part of $\rho$ n which complex conjugation acts
as $(-1)^i$ and $(-1)^{i-1},$ respectively. To this end note that
the eigenvalue $u$ of the geometric Frobenius automorphism acting on
$V'(f)$ equals $p\nu$ for the $\nu$ in (loc.\ cit.) due to the
compatibility conjecture $C_{WD}$ in \cite[2.4.3]{fontaine}, which
is known for modular forms (loc.cit., rem 2.4.6(ii)) and for Artin
motives,  and that
\[h(r):=\dim gr^r((M\otimes \rho)_{dR})=\left\{
                                                                         \begin{array}{ll}
                                                                           d, & \hbox{if $r=-1$ or $r=k-2$;} \\
                                                                           0, & \hbox{otherwise.}
                                                                         \end{array}
                                                                       \right.\]
because the de Rham realisation $M(f)(j)_{dR}$ has the following
decreasing filtration
\[M(f)(j)_{dR}^i=\left\{\begin{array}{ll}
              M(f)_{dR}, & \hbox{if $i\leq -j$;} \\
              M(f)_{dR}^1, & \hbox{if $1-j\leq i\leq k-1-j$;} \\
              0, & \hbox{if $k-j\leq i.$}
                \end{array}
                      \right.\]

see \cite[\S 11.3]{kato}. In particular, all the twists $M(f)(j),$
$1\leq j\leq k-1$ are critical. Moreover  $M(f)$ is pure of weight
$k-1$ with   Hodge decomposition of type $(k-1,0)+(0,k-1).$  The
existence of good basis $\gamma^+,$ $\gamma^-$ and $\delta$  of
$M_B^+,$ $M_B^-$ and the tangent space $t_M:=M_{dR}/M^0_{dR},$
respectively, in the sense of \cite[4.2.24]{fukaya-kato} follows
from \cite[17.5]{kato} where the dual situation is discussed, in
particular we have chosen $\Omega_\pm=\Omega(\gamma^\pm,\delta)$ in
the notation of \cite{fukaya-kato}.

Assuming the conjectures in \cite{fukaya-kato}, Conjecture
\ref{Conj3.2} is a direct consequence of theorem 4.2.22, proposition
4.3.15/16 in (loc.\ cit.) (observing that $V'(f)$ grants an infinite
residue extension of $p$ in $F_\infty/\q$) once we have seen that
$T$ induces the zero class in $K_0(\mathfrak{M}_\mathcal{H}(\G)).$
If $f$ is not CM, this follows from proposition 4.3.17 in (loc.\
cit.) while in the CM-case we give an argument in the proof of
Proposition \ref{twist}.

Now we focus on the case where $f$ is a CM-modular form. We will say
that $f$ has CM by a non-trivial (quadratic) Dirichlet character
$\epsilon$ if
\[
\epsilon(q)a_q = a_q
\]
for a set of primes $q$ of density $1$. If we write $K$ for the
quadratic extension that corresponds to $\epsilon$ we say that $f$
has CM by $K$. For example in our previous setting if we write
$f_{E}$ for the newform that corresponds to $E$ we have that $f_E$
has CM by the non-trivial quadratic character $\epsilon$ of
$Gal(K/\Q)$ as in that case $a_q=0$ for the primes that inert in
$K$. From now on our fixed modular form $f$ will have CM by some
quadratic field $K$ and we will write $\epsilon$ for the associated
character.  From \cite{ribet} (proposition 4.4 and theorem 4.5) we
know:
\begin{enumerate}
\item The field $K$ is imaginary.
\item Let $G:=Gal(F_\infty/K)$. Then ${\rho_{\lambda}}_{|_{G}}$ is
abelian.
\item $f$ is (automorphic)-induced by a Gr\"{o}ssencharacter $\psi$ over $K$ of type
$(k-1,0)$ and after fixing an embedding of $E$ in $\bar{\Q}$ we have
$L(f,s)=L(\psi,s)$. Henceforth we assume that $F$ contains
$K(\psi(\hat{K}^\times),$ where $\hat{K}$ denotes the finite adeles
of $K,$ e.g.\  we can just take $F=K(\psi(\hat{K}^\times)$ by
\cite[15.10]{kato}.

\item For the $\lambda$-adic representation attached to $\psi$ we have
\[
{\rho_{\lambda}}_{|_{G}}= \psi_{\lambda} \oplus \psi_{\lambda}^{c}
\]
where $\psi_{\lambda}$ is the $\bar{F}_{\lambda}^{\times}$-valued
$\lambda$-adic counterpart of the Gr\"{o}ssencharacter $\psi$. In
particular
$\rho_{\lambda}=Ind_{G}^{\G}\psi_{\lambda}=Ind_{G}^{\G}\psi^c_{\lambda},$
see also \cite[(15.11.2)]{kato}. Here our convention is that
$\psi_\lambda((\frak{a},K_\infty/K))=\psi(\frak{a})^{-1}$ and we
write $V(\psi)=V_{F_\lambda}(\psi)$ for the corresponding
representation space.

\item $\psi \psi^c=N_K^{k-1}(\epsilon_{f}\circ N_{K})$.
\end{enumerate}

\textbf{Assumption:} We are going to assume all along that the size
of the torsion part of $G$ is relative prime to $p$. \newline

Now we pick a prime $p$ that splits in $K$ and write $p=\p\bar{\p}$
for a prime $\p$ of $K$. We make the standard assumption that $\p$
is the prime that corresponds to the $p$-adic embedding
\[
K \hookrightarrow \bar{\Q} \hookrightarrow \bar{\Q}_p
\]
with respect to our fixed embedding $\bar{\Q} \hookrightarrow
\bar{\Q}_p,$ in particular $\lambda|\p.$ We now note that
\[
1-a_pT+p^{k-1}\epsilon_{f}(p)T^2=(1-\psi(\bar{\p})T)(1-\psi(\p)T)
\]
and claim that $\psi(\bar{\p})$ is a $\lambda$-adic unit. Indeed, as
the character $\psi$ is of type $(k-1,0)$ we have that its
$\lambda$-adic counterpart $\psi_{\lambda}$ factors through
$Gal(K(\mathfrak{f}_{\psi}\p^{\infty})/K), $ i.e.\ is it is a
character of the form
\[
\psi_{\lambda}:Gal(K(\mathfrak{f}_{\psi}\p^{\infty})/K) \rightarrow
\bar{O}^{\times}_{\lambda}
\]
where $\bar{O}_\lambda$ the ring of integers of
$\bar{F}^{\times}_{\lambda}$. But then
$\psi(\bar{\p})=\psi_{\lambda}(Frob_{\bar{\p}}) \in
\bar{O}^{\times}_{\lambda}$ and we  have $u =\psi(\bar{\p})$.

For a $G_K$-representation $\rho:G_K\to Aut(V)$ on a finite
dimensional $F_\lambda$-vector space $V$  and any Galois stable
$\O:=\O_{F_\lambda}$-lattice $T\subseteq V$ we define
\[S(T/K_\infty):=\ker\big( \H^1(G_{S_p}(K_\infty),T\otimes \q/\z)\to \Coind^{G_{\bar{\p}}}_G\H^1(K_{\infty,\nu},T\otimes \q/\z)\big)\]
where $\nu$ as before denotes any fixed place of $K_\infty$ lying
over $\bar{\p}.$ Its Pontryagin-dual
\[\mathcal{X}(T/K_\infty)=S(T/K_\infty)^\vee\cong\mathcal{X}(K_\infty)\otimes_\zp
T^*  \] is a finitely generated $\La_\O(G)$-module. We fix an
$\O$-lattice $T(\psi)^*\subseteq V(\psi)^*$ and assume that
\[T:=T(f^*)(k-1)\] coincides with \[\Ind^K_\q T(\psi)^*\subseteq
V(f)^*=V(f^*)(k-1).\] As noted above $T(\psi)^*$ is unramified at
$\bar{\p},$ whence  the free rank one $\O$-module $T(\psi^c)^*$ with
Galois action given by the complex conjugate character
$(\psi^c)^{-1}$ is unramified at $\p.$ It follows immediately that
\[T(\psi^c)^*=T' :=T'(f^*)(k-1) \mbox{ and }
T(\psi)^*=T'':=T''(f^*)(k-1)\] as $G_K$-modules.

\begin{prop}\label{twist}
There is a natural isomorphism of $\La(\G)$-modules
\[X(T/K_\infty)\cong \Ind_\G^G\left(
\mathcal{X}(K_\infty)\otimes_\zp T(\psi)\right),\] where $G$ acts
diagonally on the tensor product.
\end{prop}

\begin{proof}
 Applying the exact functor $\Coind^G_\G$ to the defining sequence
 \[\xymatrix@C=0.5cm{
   0 \ar[r] & S(T(\psi)^*/K_\infty)\ar[r]^{ } & {\H^1(G_{S_p}(K_\infty),T(\psi)^*\otimes \q/\z) }\ar[r]^{ } & {\Coind_G^{G_{\bar{\mathfrak{p}}}} \H^1(K_{\infty,\nu},T(\psi)^*\otimes \q/\z)}     }\]
  using the transitivity of
coinduction and the isomorphisms
\begin{eqnarray*}
 \Coind_\G^G\H^1(G_{S_p}(K_\infty),T(\psi)^*\otimes \q/\z) &\cong& \Hom(G_{S_p}(K_\infty),\Coind_\G^G \big(T(\psi)^*\otimes \q/\z\big)) \\
    &\cong& H^1(G_{S_p}(K_\infty),T\otimes \q/\z)
\end{eqnarray*}
as well as
\[\H^1(K_{\infty,\nu},T(\psi)^*\otimes \q/\z) \cong\H^1(K_{\infty,\nu},T''\otimes \q/\z) ,\]
  one just obtains
the defining sequence of the full Selmer group
\[\xymatrix@C=0.5cm{
   0 \ar[r] & Sel^{ord}(T/K_\infty) \ar[r]^{ } & {\H^1(G_{S_p}(K_\infty),T\otimes \q/\z) }\ar[r]^{ } & {\Coind_\G^{\G_{\bar{\mathfrak{p}}}} \H^1(K_{\infty,\nu},T''\otimes \q/\z)}.     }\]
The result now follows by taking duals. Finally, we give the
promised proof that the class of $X(T/K_\infty)$ coincides with that
of the Selmer complex used in \cite{fukaya-kato}: Let $H'$ be the
(open) maximal torsionfree pro-$p$ (abelian) subgroup of $H,$ i.e.\
$H=H'\times H''$ for some finite abelian group $H''.$ Thus we have a
natural functor from the category $\La(H')$-mod of finitely
generated $\La(H')$-modules to $\mathfrak{M}_{H}(G)$ by extending
the $H'$-action to a $G$-action letting $G/H'\cong H''\times \Gamma$
act trivially, which induces  the first homomorphism in the
following composition
\[K_0(\La(H')\mbox{-mod})\to K_0(\mathfrak{M}_{H}(G))\to K_0(\mathfrak{M}_{H}(G))\to K_0(\mathfrak{M}_{\mathcal{H}}(\G)),\]
where the second map is induced by twisting with $T(\psi)^*$ while
the last one is induced by tensoring with
$\La(\G)\otimes_{\La(G)}-.$ As, $\La(H')$ being a regular local
ring,    the $\La(H')$-rank induces an isomorphism
$K_0(\La(H')\mbox{-mod})\cong\z,$ one sees that the class of the
trivial $H'$-module $\zp,$ which is sent to the class of
$T=\Ind^G_\G(T(\psi)^*$ under the above map, is zero.
\end{proof}

Now we explain the construction of the non-abelian $p$-adic
$L$-function. We start by fixing archimedean and $p$-adic periods
that correspond canonical to the Gr\"{o}ssencharacter $\psi$ that we
have associated to our CM modular form $f$. We  pick  a prime ideal
$\mathfrak{f}$ of $K$ that is relative prime to $\p$ and with the
property that the integer $w_{\mathfrak{f}}$ defined as the number
of roots of unity in $K$ congruent to 1 modulo $\mathfrak{f}$ is
equal to 1. Then from \cite{deSh} (Lemma in page 41) we know that
there is a Gr\"{o}ssencharacter $\phi$ of $K$ of conductor
$\mathfrak{f}$ and type $(1,0)$. Moreover from the same lemma in
\cite{deSh} we know that if $\phi$ is a Gr\"{o}ssencharacter of $K$
of type $(1,0)$ then there exists an elliptic curve defined over
$K(\mathfrak{f}_\phi)$ where $\mathfrak{f}_{\phi}$ the conductor of
$\phi$ such that $E$ has CM by $O_K$ and its associated
Gr\"{o}ssencharacter is given by $\psi_{E}=\phi \circ
N_{K(\mathfrak{f}_\phi)/K}$. Now we are ready to define our periods
for our Gr\"{o}ssencharacter $\psi$. We distinguish two cases
\begin{enumerate}
\item $\psi$ is of type $(1,0)$, i.e. $f$ is of weight 2. Then we
write $E$ for the elliptic curve defined over $K(\mathfrak{f}_\psi)$
and $\Lambda$ for its corresponding lattice in $\mathbb{C}$. Then we
define $\Omega_\infty$ by
\[
\Omega_{\infty} \Lambda = \mathfrak{f}
\]
where we implicity assuming that we see $\mathfrak{f}$ as a lattice
in $\mathbb{C}$ with respect the embedding $K \hookrightarrow
\mathbb{C}$ imposed by the CM type of the character $\psi$ (see also
\cite{deSh} page 66). The $p$-adic periods $\Omega_p$ are defined
using the elliptic curve $E$ as is done in \cite{deSh} page 66)

\item If the character $\psi$ is of type $(k-1,0)$ for $k >2$ then
we pick some other Gr\"{o}ssencharacter $\phi$ of type $(1,0)$ and
of conductor $\mathfrak{f}_\phi$ prime to $\p$ and write
\[
\psi = \theta \phi^{k-1}
\]
where $\theta$ is some finite order character of conductor relative
prime to $\p$. We define then the periods $\Omega_\infty$ and
$\Omega_p$ as in (i) using the character $\phi$.
\end{enumerate}

Let $K(\p^\infty)$ be the maximal $\Z_p$-extension of $K$ inside
$K(\mathfrak{f}_\psi\p^\infty)$ and set
$G':=Gal(K(\mathfrak{f}_\psi\p^\infty)/K(\p^\infty))$. Let
$m:=|G'|$, then we define $D$ to be the ring generated over
$\widehat{\mathbb{Z}_p^{nr}}$ by the $m^{th}$ roots of unity. Using
exactly the same construction as in section 2 we conclude the
following corollary

\begin{cor}\label{p-LCM1}
There exists a unique $\mathcal{L}_{\bar\psi} : = \mu \in \La_D (G)$
such that for $0 \leq -j$ and $0<k-1+j$
\[
 \mathcal{L}_{\bar\psi} { (\chi \kappa^{j}) } = \mathop{\int}\limits_{G}
 { \chi \kappa^{j} } d\mu =\Gamma(k-1+j)i^{k-1+j}\frac{L ({\bar\psi} \chi,
 k-1+j)}{(2\pi)^j{\Omega^{k-1}_\infty}} \times
\]
\[
 e_{\mathfrak{p}}(\bar{\chi})
 P_{\frak{p}}({\bar\chi}, \frac{w}{p^{-j+1}}) P_{\bar{\frak{p}}} ({ \chi},
 u^{-1}p^{-j})\,\,\left(\frac{\psi(\mathfrak{p})}{p}\right)^{\mathfrak{f_{\mathfrak{p}}}(\bar{\chi})}p^{j\mathfrak{f_{\mathfrak{p}}}(\bar{\chi})}
\]
for all Artin-character $\chi$ of $G$ and the cyclotomic character
$\kappa : Frob^{-1}_{\mathfrak{q}} \mapsto N_K(\mathfrak{q})$.
\end{cor}
\begin{proof} We write $\mathfrak{g}$ for the maximal ideal that is contained in $\mathfrak{f}$ and $\bar{\mathfrak{f}}$.
From \ref{measureKatz} we know that there exists a measure
$\mu(\mathfrak{g}\bar{\mathfrak{p}}^\infty)$ of
$G(K(\mathfrak{g}p^\infty)/K)$ so that
\[
\Omega_p^{j-k}\int_{G(K(\mathfrak{g}p^\infty)/K)}\epsilon^{-1}(\sigma)\,\,\mu(\mathfrak{g}\bar{\mathfrak{p}}^\infty)=\Gamma(k)i^k\Omega_{\infty}^{j-k}(\frac{\sqrt{d_K}}{2\pi})^{j}G(\epsilon)(1-\frac{\epsilon(\mathfrak{p})}{p})L_{\bar{\mathfrak{p}}\mathfrak{g}}(\epsilon^{-1},0)
\]
for $\epsilon$ of type $(k,j)$ with $0 \leq -j$ and $k > 0$.  As
explained above the character $\psi$ is of the form
$\phi^{k-1}\theta$ for $\phi$ of type $(1,0)$ and $\theta$ a finite
character, both unramified at $p$. Moreover we have the relation
$\phi \bar\phi =N_K$. In particular for a finite order character
$\chi$ of $Gal(K(p^\infty)/K)$ the character
\[
\epsilon:= \psi N_{K}^{j}\bar{\chi} = \phi^{k-1+j}\bar{\phi}^j
\theta \bar{\chi}
\]
is a valid choice for $\epsilon$ above provided that $0 \leq -j$ and
$0< k-1+j$. But then for the $L$ function we have the equalities
\[
L(\epsilon^{-1},0)=L(\psi^{-1}N_K^{-j}\chi,0)=L(\bar{\psi}\chi
N_K^{-(k-1+j)})=L(\bar{\psi}\chi,k-1+j)
\]
Then the proof is the same as in the case of elliptic curves with CM
that we did in corollary \ref{measure1}.
\end{proof}
We note that in the more general case that we consider now we do not
have $\bar{\psi}=\psi^{c}$ but only that $\psi^c=\bar\psi
(\epsilon_{f}\circ N_{K})$ which means that the motive $M(f)$ is not
self-dual. We now use the functional equation in order to get the
critical values at the point $s=1$. For the character
$\psi\bar{\chi}$ of type $(k-1,0)$ we have that its dual
representation is given by the character
$(\psi\bar{\chi})^{-1}=\psi^{-1}\chi$ for which we have that
$L(\psi^{-1}\chi,s)=L(\bar{\psi}\chi,s+(k-1))$. Then the functional
equation reads (see \cite{Tate} page 16 or \cite{deSh} page 37)
\[
\Gamma_{\psi}(s)L(\psi\bar{\chi},s)=e(\psi\bar{\chi}\omega_s)\Gamma_{\bar{\psi}}((k-1)+1-s)L(\bar{\psi}\chi,(k-1)+1-s)
\]
where for a character $\phi$ of type $(k,j)$ we write
$\Gamma_{\phi}(s):=\frac{\Gamma(s-min(k,j))}{(2\pi)^{s-min(k,j)}}$
and
$e(\psi\bar{\chi}\omega_s)=\prod_{\mathfrak{q}}e_{\mathfrak{q}}(\psi\bar{\chi}\omega_s,\psi_{ad},dx_{\psi_{ad}})$
where $\omega_s$ as in Tate's \cite{Tate}. In particular for
$s:=1-j$ with $j \leq 0$ we have
\[
L ({\bar\psi} \chi,
k-1+j)=\frac{\Gamma_{\psi}(1-j)}{\Gamma_{\bar{\psi}}(k-1+j)}e(\psi\bar{\chi}\omega_{1-j})^{-1}L(\psi\bar{\chi},1-j)
\]
Hence
\[
L ({\bar\psi} \chi, k-1+j)\,
 e_{\mathfrak{p}}(\bar{\chi})\left(\frac{\psi(\mathfrak{p})}{p}\right)^{\mathfrak{f_{\mathfrak{p}}}(\bar{\chi})}p^{j\mathfrak{f_{\mathfrak{p}}}(\bar{\chi})}
=\frac{\Gamma_{\psi}(1-j)}{\Gamma_{\bar{\psi}}(k-1+j)}e(\psi\bar{\chi}\omega_{1-j})^{-1}e_{\mathfrak{p}}(\bar{\chi})\left(\frac{\psi(\mathfrak{p})}{p}\right)^{\mathfrak{f_{\mathfrak{p}}}(\bar{\chi})}p^{j\mathfrak{f_{\mathfrak{p}}}(\bar{\chi})}
L(\psi\bar{\chi},1)
\]
But we have
\[
e(\psi\bar{\chi}\omega_{1-j})^{-1}e_{\mathfrak{p}}(\bar{\chi})\left(\frac{\psi(\mathfrak{p})}{p}\right)^{\mathfrak{f_{\mathfrak{p}}}(\bar{\chi})}p^{j\mathfrak{f_{\mathfrak{p}}}(\bar{\chi})}
=\prod_{\mathfrak{q}}e_{\mathfrak{q}}(\psi\bar{\chi}\omega_{1-j},\psi_{ad},dx_{\psi_{ad}})^{-1}e_{\mathfrak{p}}(\psi\bar{\chi}\omega_{1-j})=
\prod_{\mathfrak{q}\neq
\mathfrak{p}}e_{\mathfrak{q}}(\psi\bar{\chi}\omega_{1-j},\psi_{ad},dx_{\psi_{ad}})^{-1}
\]
where we have used the fact that $dx_{\psi_{ad}}=dx_1$ at
$\mathfrak{p}$ as this prime is unramified in $K$. Now we observe
that
\[
e_{\bar{\mathfrak{p}}}(\psi\bar{\chi}\omega_{1-j},\psi)=e_{\bar{\mathfrak{p}}}(\bar{\chi}\omega_{1-j},\psi_{ad},dx_{1})\psi(\bar{\mathfrak{p}})^{\mathfrak{f_{\bar{\mathfrak{p}}}}(\bar{\chi})}=
e_{\bar{\mathfrak{p}}}(\chi\omega_j,\psi_{ad}^{-1})^{-1}\psi(\bar{\mathfrak{p}})^{\mathfrak{f_{\bar{\mathfrak{p}}}}(\bar{\chi})}=e_{\bar{\mathfrak{p}}}(\chi\omega_j,\psi_{ad}^{-1})^{-1}u^{\mathfrak{f_{\bar{\mathfrak{p}}}}(\bar{\chi})}
\]
where we have used the duality
$e_{\bar{\mathfrak{p}}}(\chi\omega_j,\psi_{ad}^{-1},dx_1)e_{\bar{\mathfrak{p}}}(\bar{\chi}\omega_{1-j},\psi_{ad},dx_1)=1$.
Hence we obtain
\[
\prod_{\mathfrak{q}\neq
\mathfrak{p}}e_{\mathfrak{q}}(\psi\bar{\chi}\omega_{1-j},\psi_{ad},dx_{\psi_{ad}})^{-1}=
\prod_{\mathfrak{q}\neq
\mathfrak{p},\bar{\mathfrak{p}}}e_{\mathfrak{q}}(\psi\bar{\chi}\omega_{1-j},\psi_{ad},dx_{\psi_{ad}})^{-1}e_{\bar{\mathfrak{p}}}(\chi)u^{-\mathfrak{f_{\bar{\mathfrak{p}}}}(\bar{\chi})}p^{-j\mathfrak{f_{\mathfrak{p}}}(\bar{\chi})}
\]
We can now state the following corollary
\begin{cor} \label{p-LCM2}
There exists a unique $\mathcal{L}^{(tw)}_{\bar\psi} : = \mu \in
\La_D (G)$ such that for $0 \leq -j$ and $0<k-1+j$ we have
\[
 \mathcal{L}^{(tw)}_{\bar\psi} { (\chi \kappa^j) } = \mathop{\int}\limits_{G}
 { \chi\kappa^j } d\mu =\Gamma(k-1+j)i^j\frac{L (\psi \bar{\chi}, 1-j)}{(2\pi)^j{\Omega^{k-1}_\infty}} \,
 e_{\bar{\mathfrak{p}}}(\chi) \times
 \]
\[
 P_{\frak{p}}({\bar\chi}, \frac{w}{p^{-j+1}}) P_{\bar{\frak{p}}} ({ \chi},
 u^{-1}p^{-j})\,\,u^{-\mathfrak{f_{\bar{\mathfrak{p}}}}(\bar{\chi})}p^{-j\mathfrak{f_{\mathfrak{p}}}(\bar{\chi})}\left(\frac{\Gamma_{\psi}(1-j)}{\Gamma_{\bar\psi}(k-1+j)}\right)
\]
for all Artin-character $\chi$ of $G.$
\end{cor}
\begin{proof}We have already constructed a measure with the
interpolation property
\[
 \mathcal{L}_{\bar\psi} { (\chi \kappa^{j}) } = \mathop{\int}\limits_{G}
 { \chi \kappa^{j} } d\mu =\Gamma(k-1+j)\frac{L ({\bar\psi} \chi, k-1+j)}{(2\pi)^j{\Omega^{k-1}_\infty}}
 e_{\mathfrak{p}}(\bar{\chi})
 P_{\frak{p}}({\bar\chi}, \frac{w}{p^{-j+1}}) P_{\bar{\frak{p}}} ({ \chi},
 u^{-1}p^{-j})\,\,\left(\frac{\psi(\mathfrak{p})}{p}\right)^{\mathfrak{f_{\mathfrak{p}}}(\bar{\chi})}p^{j\mathfrak{f_{\mathfrak{p}}}(\bar{\chi})}
\]
and using the above computations we can rewrite it as
\[
 \mathcal{L}_{\bar\psi} { (\chi \kappa^j) } = \mathop{\int}\limits_{G}
 { \chi \kappa^j } d\mu = \left(\frac{\Gamma_{\psi}(1-j)}{\Gamma_{\bar\psi}(k-1+j)}
\prod_{\mathfrak{q}\neq
\mathfrak{p},\bar{\mathfrak{p}}}e_{\mathfrak{q}}(\psi\bar{\chi}\omega_{1-j},\psi_{ad},dx_{\psi_{ad}})^{-1}\right)\times
\]
\[
\Gamma(k-1+j)i^j\frac{L (\psi \bar{\chi},
1-j)}{(2\pi)^j{\Omega^{k-1}_\infty}} \,
 e_{\bar{\mathfrak{p}}}(\chi)
 P_{\frak{p}}({\bar\chi}, \frac{w}{p^{-j+1}}) P_{\bar{\frak{p}}} ({ \chi},
 u^{-1}p^{-j})\,\,u^{-\mathfrak{f_{\bar{\mathfrak{p}}}}(\bar{\chi})}p^{-j\mathfrak{f_{\mathfrak{p}}}(\bar{\chi})}
\]
Now we claim that the mapping $\chi\kappa^j \mapsto
\prod_{\mathfrak{q}\neq
\mathfrak{p},\bar{\mathfrak{p}}}e_{\mathfrak{q}}(\psi\bar{\chi}\omega_{1-j},\psi_{ad},dx_{\psi_{ad}})$
is a unit in the Iwasawa algebra $\Lambda_D(G)$ and hence we can
twist our measure by this element to conclude the proposition.

We write $K_{\infty} \subseteq K(p^\infty)$ for the
$\z_p^2$-extension of $K$ and define $\Gamma:=Gal(K_{\infty}/K)
\cong \z^2_p$. We decompose $G=\Delta \times \Gamma$ for $\Delta$
finite of order relative prime to $p$. Then we have that
$\Lambda_D(G)=D[[G]]=D[\Delta][[\Gamma]]$. We write $\hat{\Delta}$
for the group of characters of $\Delta$. Then by our assumptions on
$D$ we have that $D[\Delta] \cong \oplus_{\theta \in \hat{\Delta}}
D$ given by $\alpha \mapsto (\ldots,\theta(\alpha),\ldots)$ and
hence $\Lambda_D(G) \cong \oplus_{\theta}D[[\Gamma]]$. Every Artin
character $\chi$ of $G$ can be decomposed as $\chi=\theta
\chi_{\Gamma}$ for $\theta$ a character of $\Delta$ and
$\chi_{\Gamma}$ a character of $\Gamma$. We write $\mathfrak{d}_K$
for the different of $K$ over $\Q$ We note that $\theta$ can be
ramified at $p$ or at $\mathfrak{q}$ with $\mathfrak{q} |
\mathfrak{f}_\psi$ for the conductor of $\psi$. We write
$\mathfrak{f}_{\psi\theta}$ for the non-$p$ part of the conductor of
$\psi\bar{\theta}$. Then we have
\[
\prod_{\mathfrak{q}\neq
\mathfrak{p},\bar{\mathfrak{p}}}e_{\mathfrak{q}}(\psi\bar{\chi}\omega_{1-j},\psi,dx_{\psi})=i^{k-1}\prod_{\mathfrak{q}
|
\mathfrak{f}_{\psi}\mathfrak{d}_K}e_{\mathfrak{q}}(\psi\bar{\chi}\omega_{1-j},\psi_{ad},dx_{\psi_{ad}})=
\]
\[
\bar{\chi}_\Gamma(\mathfrak{f}_{\psi\theta}\mathfrak{d}_K)N_K(\mathfrak{f}_{\psi\theta}\mathfrak{d}_K)^{-(1-j)}i^{k-1}\prod_{\mathfrak{q}
|
\mathfrak{f}_{\psi}\mathfrak{d}_K}e_{\mathfrak{q}}(\psi\bar{\theta},\psi_{ad},dx_{\psi_{ad}})
\]
since $\chi_\Gamma$ is only at $p$ ramified and
$(p,\mathfrak{f}_{\psi}\mathfrak{d}_K)=1$ under our assumptions. But
$\prod_{\mathfrak{q}|\mathfrak{f}_{\psi}\mathfrak{d}_K}e_{\mathfrak{q}}(\psi\bar{\theta},\psi_{ad},dx_{\psi_{ad}})$
is a $\mathfrak{p}$-adic unit (see \cite{deSh} page 94) and hence we
can define the measure
\[
\frac{1}{N_K(\mathfrak{f}_{\psi\theta}\mathfrak{d}_K)}\prod_{\mathfrak{q}
|
\mathfrak{f}_{\psi}\mathfrak{d}_K}e_{\mathfrak{q}}(\psi\bar{\theta},\psi_{ad},dx_{\psi_{ad}})\sigma^{-1}_{\mathfrak{f}_{\psi\theta}\mathfrak{d}_K}
\in \La_D [[\Gamma]]^{\times}
\]
where $\sigma_{\mathfrak{f}_{\psi\theta}\mathfrak{d}_K } \in \Gamma$
corresponds through the Artin reciprocity to
$\mathfrak{f}_{\psi\theta}\mathfrak{d}_K$, well defined as
$K_{\infty}$ ramifies only at $p$. We then define the element $E \in
\La_D[[G]]^{\times }$ to be the element that corresponds to
\[
(\ldots,\frac{1}{N_K(\mathfrak{f}_{\psi\theta}\mathfrak{d}_K)}\prod_{\mathfrak{q}
|
\mathfrak{f}_{\psi}\mathfrak{d}_K}e_{\mathfrak{q}}(\psi\bar{\theta},\psi_{ad},dx_{\psi_{ad}})\sigma^{-1}_{\mathfrak{f}_{\psi\theta}\mathfrak{d}_K},\ldots
) \in \oplus_{\theta}D[[\Gamma]]
\]
under the above mentioned isomorphism $\Lambda_D(G) \cong
\oplus_{\theta}D[[\Gamma]]$. We then obtain for a character $\chi$
of $G$
\[
E(\chi \kappa^j) =i^{1-k} \prod_{\mathfrak{q}\neq
\mathfrak{p},\bar{\mathfrak{p}}}e_{\mathfrak{q}}(\psi\bar{\chi}\omega_{1-j},\psi,dx_{\psi})
\]
which it allows us to conclude the proposition.
\end{proof}

Using now the natural map
\[
(\imath_{S})_{*}:K_{1}(\Lambda(G)_{S^*}) \rightarrow
K_1(\Lambda(\G)_{S^*})
\]
we define
\[
\mathcal{L}:=(\imath_S)_*(\mathcal{L}_{\bar\psi})\,\,\,and\,\,\,\mathcal{L}^{(tw)}:=(\imath_S)_*(\mathcal{L}^{(tw)}_{\bar\psi}).
\]
As in the case of elliptic curves with CM we need to understand the
following correction term which describes the change of complex
periods
\[\mathcal{L}_\Omega:= \frac{(2\pi)^{k-2}\Omega_+}{\Omega^{k-1}_\infty} \frac{1+c}{2} + \frac{(2\pi)^{k-2}\Omega_-}{\Omega^{k-1}_\infty}\frac{1-c}{2}.\]

\begin{conj}[Period Relation]\label{ConjPer} We conjecture
\[
\frac{(2\pi)^{k-2}\Omega_+}{\Omega^{k-1}_\infty} \in
O_{E_\lambda}^\times \,\,\,and
\,\,\,\frac{(2\pi)^{k-2}\Omega_-}{\Omega^{k-1}_\infty} \in
O_{E_\lambda}^\times
\]
and hence   $\mathcal{L}_\Omega\in \La_{O_{E_\lambda}}(\G)^\times.$
\end{conj}
The difficulty in proving the above conjecture for $k \geq 3$ is due
to the fact that it is not clear whether the motive $M(f)$ that it
is attached to $f$ (a direct summand of the motive associated to a
Kuga-Sato variety) coincides with the motive that we have associated
to the Gr\"{o}ssencharacter $\psi$ as $k-1$-fold tensor power of the
elliptic curve associated to the character $\phi$ of type $(1,0)$.
For a similar discussion see also \cite[page 263]{Kato}.

We define
\[
\mathcal{L}'_{f_{\psi}}:=\mathcal{L}\mathcal{L}^{-1}_{\Omega}\,\,\,and\,\,\,{\mathcal{L}^{(tw)}}'_{f_{\psi}}:=\mathcal{L}^{(tw)}\mathcal{L}^{-1}_{\Omega}
\]
Assuming the above conjecture we have
\begin{prop}Let $M(f)^{\vee}$ be the dual motive to $M(f)$. Then
$\mathcal{L}'_{f_{\psi}}$ satisfies the following interpolation
property for $0 \leq -j$ and $0<k-1+j$.
\[
\mathcal{L}'_{f_{\psi}} (\rho \kappa^j) =
\Gamma(k-1+j)^{d(\rho)}i^{d(\rho)j}\frac{L_{\{p\}}(M(f)^{\vee},
\rho, k-1+j)}{(2\pi)^{2(k-2)+2j}\Omega_+^{d^+ (\rho)}
\Omega_-^{d^-(\rho)}}\, e_p(\check{\rho})
\frac{P_p(\mathop{\rho}\limits^{\vee}, \frac{w}{p^{-j+1}})}
{P_p(\rho, w^{-1}p^{-j})} \left(\frac{w}{p}\right)^{{\frak{f}}_p
(\check{\rho})}p^{j\mathfrak{f}_p(\check{\rho})}
\]
for all Artin representations $\rho$ of $\G.$
\end{prop}
\begin{proof}
As in the proof in the elliptic curve case we consider two cases of
artin representations, those that are one-dimensional and those that
are induced from a character from $K$ to $\mathbb{Q}$. We explain
here the case where $\rho=Ind\chi$ to indicate the similarities and
differences with above. We compute
\[
\mathcal{L}(\rho \kappa^j)=\mathcal{L}_{\bar{\psi}}(\chi
\kappa^j)\mathcal{L}_{\bar{\psi}}(\chi^{c}\kappa^j)=
\]
\[
\Gamma(k-1+j)^2i^{2j}\frac{L ({\bar\psi} \chi,
k-1+j)L(\bar{\psi}\chi^c,k-1+j)}{(2\pi)^{2j}{\Omega^{2(k-1)}_\infty}}
\,
 e_{\mathfrak{p}}(\bar{\chi})e_{\bar{\mathfrak{p}}}(\bar{\chi})
 \,\,\left(\frac{\psi(\mathfrak{p})p^j}{p}\right)^{\mathfrak{f_{\mathfrak{p}}}(\bar{\chi})+\mathfrak{f_{\bar{\mathfrak{p}}}}(\bar{\chi})}\times
 \]
\[
 P_{\frak{p}}({\bar\chi}, \frac{w}{p^{-j+1}}) P_{\bar{\frak{p}}} ({ \chi},
 u^{-1}p^{-j})P_{\bar{\frak{p}}}({\bar\chi}, \frac{w}{p^{-j+1}}) P_{\frak{p}} ({ \chi},
 u^{-1}p^{-j})
\]
\[
=\Gamma(k-1+j)^2i^{2j}\frac{L_{\{\p,\bar{\p}\}}({\bar\psi} \chi,
k-1+j)L_{\{\p,\bar{\p}\}}(\bar{\psi}\chi^c,k-1+j)}{(2\pi)^{2j}{\Omega^{2(k-1)}_\infty}}
\,
 e_{\mathfrak{p}}(\bar{\chi})e_{\bar{\mathfrak{p}}}(\bar{\chi})
 \,\,\left(\frac{\psi(\mathfrak{p})p^j}{p}\right)^{\mathfrak{f_{\mathfrak{p}}}(\bar{\chi})+\mathfrak{f_{\bar{\mathfrak{p}}}}(\bar{\chi})}\times\]
\[
\frac{P_{\frak{p}}({\bar\chi}, \frac{w}{p^{-j+1}})
P_{\bar{\frak{p}}} ({ \chi},
 u^{-1}p^{-j})P_{\bar{\frak{p}}}({\bar\chi}, \frac{w}{p^{-j+1}}) P_{\frak{p}} ({ \chi},
 u^{-1}p^{-j})}{P_{\p}(\chi,w^{-1}p^{-j})P_{\bar{\p}}(\chi,u^{-1}p^{-j})P_{\bar{\p}}(\chi,w^{-1}p^{-j})P_{\p}(\chi,u^{-1}p^{-j})}
\]
\[
=\Gamma(k-1+j)^2i^{2j}\frac{L_{\{\p,\bar{\p}\}}({\bar\psi} \chi,
k-1+j)L_{\{\p,\bar{\p}\}}(\bar{\psi}\chi^c,k-1+j)}{(2\pi)^{2j}{\Omega^{2(k-1)}_\infty}}
\,
 e_{\mathfrak{p}}(\bar{\chi})e_{\bar{\mathfrak{p}}}(\bar{\chi})
 \,\,\left(\frac{\psi(\mathfrak{p})p^j}{p}\right)^{\mathfrak{f_{\mathfrak{p}}}(\bar{\chi})+\mathfrak{f_{\bar{\mathfrak{p}}}}(\bar{\chi})}
 \]
 \[
 \times \frac{P_{\frak{p}}({\bar\chi}, \frac{w}{p^{-j+1}})
P_{\bar{\frak{p}}}({\bar\chi}, \frac{w}{p^{-j+1}})
}{P_{\p}(\chi,w^{-1}p^{-j})P_{\bar{\p}}(\chi,w^{-1}p^{-j})}
\]
\[
=e_{p}(\check{\rho})\left(\frac{\psi(\mathfrak{p})p^j}{p}\right)^{\mathfrak{f}_p(\rho)}\Gamma(k-1+j)^2i^{2j}\frac{L_{\{p\}}(M(f)^{\vee},
\rho,
k-1+j)}{(2\pi)^{2j}{\Omega^{2(k-1)}_\infty}}\frac{P_{p}({\check\rho},
\frac{w}{p^{-j+1}})}{P_{p}(\rho,w^{-1}p^{-j})}
\]
The last equation follows from the fact that since $\psi$ induces
the cusp form $f_{\psi}=f=\sum_{n \geq 1}a_n q^n$, the character
$\bar{\psi}$ induces a cuspidal newform $f_{\bar\psi}=\sum_{n\geq
1}\overline{a_n}q^n$ where the complex conjugation on the
coefficients $a_n$ is with respect to our fixed embedding $F
\hookrightarrow \bar{\Q}$. Actually one has that
$\overline{a_n}=\epsilon_f^{-1}(n)a_n$. But then $f_{\bar{\psi}}$
corresponds to the dual motive $M(f)^{\vee}$. Then as above we
conclude that
$\mathcal{L}'_{f_{\psi}}=\mathcal{L}\mathcal{L}^{-1}_{\Omega}$ has
the claimed interpolation property.
\end{proof}

Similarly we have
\begin{prop}The measure
${\mathcal{L}^{(tw)}}'_{f_{\psi}}$ satisfies the following
interpolation property for $0 \leq -j$ and $0< k-1+j$,
\[
{\mathcal{L}^{(tw)}}'_{f_{\psi}} (\rho \kappa^j) =
\Gamma(1-j)^{d(\rho)}i^{d(\rho)j}\frac{L_{\{p\}}(M(f), \check{\rho},
1-j)}{(2\pi)^{-2j}\Omega_+^{d^+ (\rho)} \Omega_-^{d^-(\rho)}}\,
e_p(\rho) \frac{P_p(\rho, u^{-1}p^{-j})} {P_p(\check{\rho},
\frac{u}{p^{-j+1}})} (up^j)^{-{\frak{f}}_p (\check{\rho})}
\]
for all Artin representations $\rho$ of $\G.$
\end{prop}
\begin{proof}As above we have
\[
\mathcal{L}^{(tw)}(\rho
\kappa^j)=\mathcal{L}^{(tw)}_{\bar{\psi}}(\chi
\kappa^j)\mathcal{L}^{(tw)}_{\bar{\psi}}(\chi^{c}\kappa^j)=
\]
\[
\Gamma(k-1+j)^2i^{2j}\frac{L (\psi \bar\chi,
1-j)L(\psi\bar\chi^c,1-j)}{(2\pi )^{2j}{\Omega^{2(k-1)}_\infty}} \,
 e_{\mathfrak{p}}(\chi)e_{\bar{\mathfrak{p}}}(\chi)
 \,\,(up^j)^{\mathfrak{-f_{\mathfrak{p}}}(\bar{\chi})-\mathfrak{f_{\bar{\mathfrak{p}}}}(\bar{\chi})}
 \]
 \[
 \times P_{\frak{p}}({\bar\chi}, \frac{w}{p^{-j+1}}) P_{\bar{\frak{p}}} ({ \chi},
 u^{-1}p^{-j})P_{\bar{\frak{p}}}({\bar\chi}, \frac{w}{p^{-j+1}}) P_{\frak{p}} ({ \chi},
 u^{-1}p^{-j})\left(\frac{\Gamma_{\psi}(1-j)}{\Gamma_{\psi}(k-1+j)}\right)^{2}
\]
\[
=\Gamma(k-1+j)^2i^{2j}\frac{L_{\{\p,\bar{\p}\}}(\psi \bar{\chi},
1-j)L_{\{\p,\bar{\p}\}}(\psi \bar{\chi}^c,1-j)}{(2\pi
)^{2j}{\Omega^{2(k-1)}_\infty}} \,
 e_{\mathfrak{p}}(\chi)e_{\bar{\mathfrak{p}}}(\chi)
 \,\,(up^j)^{-\mathfrak{f_{\mathfrak{p}}}(\bar{\chi})-\mathfrak{f_{\bar{\mathfrak{p}}}}(\bar{\chi})}\times
 \]
\[
\frac{P_{\frak{p}}({\bar\chi}, \frac{w}{p^{-j+1}})
P_{\bar{\frak{p}}} ({ \chi},
 u^{-1}p^{-j})P_{\bar{\frak{p}}}({\bar\chi}, \frac{w}{p^{-j+1}}) P_{\frak{p}} ({ \chi},
 u^{-1}p^{-j})}{P_{\p}(\bar{\chi},\frac{w}{p^{-j+1}})P_{\bar{\p}}(\bar{\chi},\frac{u}{p^{-j+1}})P_{\bar{\p}}(\bar{\chi},\frac{w}{p^{-j+1}})P_{\p}(\bar{\chi},\frac{u}{p^{-j+1}})}
\left(\frac{\Gamma_{\psi}(1-j)}{\Gamma_{\bar\psi}(k-1+j)}\right)^{2}
\]
\[
=\Gamma(k-1+j)^2i^{2j}\frac{L_{\{\p,\bar{\p}\}}({\psi} \bar\chi,
1-j)L_{\{\p,\bar{\p}\}}({\psi} \bar\chi^c,1-j)}{(2\pi
)^{2j}{\Omega^{2(k-1)}_\infty}} \,
 e_{\mathfrak{p}}(\chi)e_{\bar{\mathfrak{p}}}(\chi)
 \,\,(up^j)^{\mathfrak{-f_{\mathfrak{p}}}(\bar{\chi})-\mathfrak{f_{\bar{\mathfrak{p}}}}(\bar{\chi})}
 \]
 \[ \times \frac{P_{\frak{p}}({\chi}, u^{-1}p^{-j})
P_{\bar{\frak{p}}}({\chi}, u^{-1}p^{-j})
}{P_{\p}(\bar{\chi},\frac{u}{p^{-j+1}})P_{\bar{\p}}(\bar\chi,\frac{u}{p^{-j+1}})}\left(\frac{\Gamma_{\psi}(1-j)}{\Gamma_{\bar\psi}(k-1+j)}\right)^2
\]
\[
=e_{p}(\rho)(up^j)^{-\mathfrak{f}_p(\check{\rho})}\Gamma(k-1+j)^2i^{2j}\frac{L_{\{p\}}(M(f),
\check{\rho},
1-j)}{(2\pi)^{2j}{\Omega^{2(k-1)}_\infty}}\frac{P_{p}({\rho},u^{-1}p^{-j})}{P_{p}(\check{\rho},\frac{u}{p^{-j+1}})}\left(\frac{\Gamma_{\psi}(1-j)}{\Gamma_{\bar\psi}(k-1+j)}\right)^2
\]

We recall that
$\Gamma_{\phi}(s)=\frac{\Gamma(s-min(k,j))}{(2\pi)^{s-min(k,j)}}$
for $\phi$ a Gr\"{o}ssencharacter of type $(k,j)$. In particular for
the character $\psi$ of type $(k-1,0)$ (hence $\bar\psi$ of type
$(0,k-1)$) we have that
\[
\frac{\Gamma_{\psi}(1-j)}{\Gamma_{\bar\psi}(k-1+j)}=\frac{\Gamma(1-j)}{\Gamma(k-1+j)}
(2\pi)^{k-2+2j}
\]
This allows us to conclude the proof of the proposition.
\end{proof}

Altogether we obtain the following (under our assumption over the
torsion part of $G$),

\begin{thm}
Assuming Conjectures \ref{torConj2} and \ref{ConjPer} there exists
$\mathcal{L}'_{M(f)}\in K_1(\La_D(\G)_S)$ satisfying Conjecture
\ref{Conj2.2} and \ref{Conj3.2}.
\end{thm}

\section{Appendix: Galois descent for Iwasawa algebras}

The aim of this appendix is to provide some results for Galois
descent of (non-commutative) $p$-adic $L$-functions, by which we
mean the following general question:

Assume that $G$ is a compact $p$-adic Lie group which possesses an
closed normal subgroup $H$ such that $G/H=:\Gamma\cong\zp$ and let
$\O=\O_L$ denote the ring of integers of a finite extension $L$ over
$\qp.$ Then we denote by $D$ the discrete valuation ring of the
completion $\tilde{L}$ of the maximal unramified extension $L^{nr}$
of $L.$ If $\mathcal{L}\in K_1(\Lambda_D(G)_{S^*})$ has the property
\[\L(\rho)\in \O(\rho) \mbox{ for all } \rho\in \mathrm{Irr}G,\]
does it hold that there exists an element $\L'\in
K_1(\Lambda_\O(G)_{S^*})$ with \[\L'(\rho)=\L(\rho)\mbox{ for all }
\rho\in \mathrm{IrrG}?\]

Here $\mathrm{Irr}G$ denotes the set of isomorphism classes of
irreducible Artin representations of $G$ with values in
$\overline{\qp}$ while for every $\rho\in \mathrm{Irr}G$ we write
$\O(\rho)$ for the discrete valuation ring of
$L(\rho)=\overline{\qp}^{G_{L,\rho}},$ where $G_{L,\rho}$ is the
stabilizer subgroup of $\rho$ in $G_L=G(\overline{\qp}/L)$ with
respect to the natural action on the coefficients of $\rho.$

We start with a simple observation in the abelian case:

\begin{lem}\label{fix}
Let $G$ be a topological finitely generated virtual pro-$p,$ abelian
group, such that the exponent of its torsion part divides $p-1.$
Assume that for a given $f\in \La_D(G)$ the values $f(\rho)$ belongs
to $L(\rho)$ (and hence to $\O(\rho)$),  for all (irreducible) Artin
characters $\rho$ of $G.$ Then $f$ is already contained in
$\La_\O(G).$
\end{lem}

\begin{proof}
With out loss of generality me may assume that
$G\cong\mathbb{Z}_p^d$ for some natural number $d,$ because
otherwise me can decompose $\La_D(G)$ into a product of such rings
by our assumption on the torsion part of $G.$ Also by taking inverse
limits afterwards we may and do assume that $G$ is a
 finite $p$-group. Then it is  well-known (e.g. a variant of prop.\ 4.5.14 in \cite{snaith}) that we have a $G_L$-invariant isomorphism
 \[\phi: \Co_p[G]\cong \Hom(R(G),\Co_p),\;\; f=\sum m_g g\mapsto \left(\rho\mapsto f(\rho)=\sum m_g \rho(g)\right),\]
 where on the left hand side the Galois action is just an the coefficients while on the right hand
 side  $h^g(\rho)=gh(g^{-1}\rho)$ for all $g\in G_\qp,$ and $h$ any homomorphism from the group of
 virtual representations $R(G):=R(G)$ of $G$ defined over $\overline{\qp}.$  Recall that by the
 Ax-Sen-Tate theorem
 \[\Co_p^H=\tilde{L}:\]
 for $H:=G(\overline{\qp})/L^{nr}_p).$ Similarly, if $H_\rho$ denotes the stabiliser of $\rho$
 we get

\[\tilde{L}(\rho):=\Co_p^{H_\rho}=\widehat{\overline{\qp}^{H_\rho}}=\widehat{L^{nr}(\rho)}.\]

Taking $H$- and $G_L$-invariants of $\phi$ thus induces the
following commutative diagram \xymatrix{
  {\tilde{L}[G]}   \ar@{=}[r]  & {\Co_p[G]^H} \ar@{=}[r] & {\Hom_H(R(G),\Co_p)} \ar@{=}[r] & {\prod_{\rho\in (\mathrm{Irr}G)/H}\tilde{L}(\rho) }  \\
 {L[G]} \ar@{=}[r]\ar@{^(->}[u] & {\Co_p[G]^{G_L}} \ar@{=}[r] & {\Hom_{G_L}(R(G),\Co_p)} \ar@{=}[r] & {\prod_{\rho\in (\mathrm{Irr}G)/G_L}L(\rho) } \ar@{^(->}[u]  }

where   $(\mathrm{Irr}G)/U$   denotes the  Galois-orbits with
respect to some closed subgroup $U\subseteq G_L.$   Since
$L(\rho)/L$ is totally ramified, it is easy to see that
\[(\mathrm{Irr}G)/H=(\mathrm{Irr}G)/G_L.\]

Now, by assumption $f\in D[G]\subseteq \tilde{L}[G]$ satisfies
$f(\rho)\in \L(\rho)$ for all $\rho\in \mathrm{Irr}G,$ i.e.\
$\phi(f)\in \prod_{\rho\in (\mathrm{Irr}G)/G_L}L(\rho).$ Hence $f\in
L[G]\cap D[G]=\O[G],$ because $L\cap D=\{x\in L|\; |x|_p\leq 1\}=\O$
coefficientwise.
\end{proof}

In the following we will use Fr\"{o}hlich's Hom-description as it
has been adapted to Iwasawa theory by Ritter and Weiss
\cite[\S3]{ritter-weiss}. We have the following commutative diagram

\[\xymatrix{
  K_1(\La_D(G))   \ar[r]^{\mathrm{Det}\phantom{mmm}} & {\Hom_H(R(G),\O_{\mathbb{C}_p}^\times)}  \\
  K_1(\La_\O(G))   \ar[r]^{\mathrm{Det}\phantom{mmm}}\ar[u]  & {\Hom_{G_L}(R(G),\O_{\mathbb{C}_p}^\times),} \ar@{^(->}[u]  }\]

  where the homomorphism Det is defined as follows: for $\rho\in \mathrm{Irr}G$ we obtain a
  homomorphism of rings \[ \rho:\La_D(G)\to M_{n_\rho}(\overline{\zp}),\] taking $K_1(-)$ of which gives a group
  homomorphism \[\psi_\rho:K_1(\La_D(G))\to K_1(M_{n_\rho}(\overline{\zp}))\cong K_1( \overline{\zp}) \subseteq \O_{\mathbb{C}_p}^\times.
  \] Now we set $\mathrm{Det}(\L)(\rho):=\psi_\rho(\L).$

  Setting $\Delta:=G(L^{nr}/L)$ we obtain the following generalisation of a theorem of M. Taylor \cite[\S 8, thm. 1.4]{taylor84}:

  \begin{thm}For $\O$ unramified over $\zp$ we have
  \[\mathrm{Det}(K_1(\La_D(G)))^{\Delta}=\mathrm{Det}(K_1(\La_\O(G))).\]
  \end{thm}

\begin{proof}
In Taylor's theorem (loc.\ cit.) $G$ is finite and $D$ is the
valuation ring of a finite unramified extension of $L.$ His proof
generalizes immediately to the case of our more general $D,$ see
\cite{iz-ven} for details, thus we only have to show how the general
case can be reduced to the case of finite groups. To this end write
$G=\projlim{n} G_n$ as inverse limit of finite groups. By Taylor's
result we have compatible continuous maps
\[\xymatrix{
  {\O[G_n]^\times }\ar@{->>}[r]^{\mathrm{Det}\phantom{mmmm}} & {\mathrm{Det}(K_1(\La_D(G_n)))^{\Delta}\phantom{m} }\ar@{^(->}[r]  & {\Hom_H(R(G_n),\O_{\mathbb{C}_p}^\times)^{\Delta}}  }\]
  where the topology on $\Hom_H(R(G_n),\O_{\mathbb{C}_p}^\times)\cong \prod_{\rho
  \in \mathrm{Irr}G_n/H} (\O_{\mathbb{C}_p}^\times)^{H_\rho}$ is induced from the valuation topology on
  $\Co_p.$
  Taking the inverse limit yields, by the compactness of $\La_\O(G)^\times=\projlim{n}(
  \O/\pi^n[G_n])^\times$  and by letting $R(G):=\dirlim{n} R(G_n)$ denote the free abelian group on the isomorphism classes of irreducible Artin representations of $G,$ a factorization of the homomorphism Det into
\[\xymatrix{
  {\La_\O(G)^\times }\ar@{->>}[r]^{\mathrm{Det}\phantom{mmmmm}} & {(\projlim{n} \mathrm{Det}(K_1(\La_D(G_n))))^{\Delta}\phantom{m} }\ar@{^(->}[r]  & {\Hom_{G_L}(R(G),\O_{\mathbb{C}_p}^\times).}  }\]
  The claim follows because   denoting by $res_n:\Hom_H(R(G),\O_{\mathbb{C}_p}^\times)\to \Hom_H(R(G_n),\O_{\mathbb{C}_p}^\times) $ the
  restriction we obtain from the universal mapping property for
  \[\projlim{n}\Hom_H(R(G_n),\O_{\mathbb{C}_p}^\times)\cong
  \Hom_H(R(G),\O_{\mathbb{C}_p}^\times)\] the inclusions
\begin{eqnarray*}
\mathrm{Det}(K_1(\La_D(G)))&\subseteq & \projlim{n} \mathrm{im}(res_n\circ \mathrm{Det}) \\
 &\subseteq &\projlim{n}   \mathrm{Det}(K_1(\La_D(G_n))),
\end{eqnarray*}
whence the obvious inclusion
\[\mathrm{Det}(K_1(\La_\O(G)))\subseteq \mathrm{Det}(K_1(\La_D(G)))^{\Delta}\subseteq (\projlim{n} \mathrm{Det}(K_1(\La_D(G_n))))^{\Delta}\]
is surjective.
\end{proof}

There are (at least) two  obvious questions: Firstly whether   this
descent result does also hold for the full $K_1$-groups ( this
amounts to an analogous statement for the $SK_1$-terms) and secondly
whether the analogue for the localisations holds, too. We shall at
least show a weak version towards the second issue. To this end we
give a variant of Ritter and Weiss' Hom-description: Consider the
following commutative diagram

\[\xymatrix{
   K_1(\Lambda_D(G) )\ar[d]_{\pi} \ar[r]^{\det\phantom{mmmmm}} & {\Hom}_{H,R_\Gamma}(R(G),\La_{\O_{\mathbb{C}_p}}(\Gamma)^\times) \ar@{^(->}[d]_{ } \ar@{^(->}[r]^{ev } & {\Hom_H}(R(G),\O_{\mathbb{C}_p}^\times) \ar@{^(->}[d]^{} \\
   K_1(\Lambda_D(G)_{S^*}) \ar[r]^{\det\phantom{mmmmm}} & {\Hom}_{H,R_\Gamma}(R(G),Q_{\O_{\mathbb{C}_p}}(\Gamma)^\times) \ar[r]^{ } & {\mathrm{Maps}}(\mathrm{Irr}G,\Co_p\cup\{\infty\}),   }\]

where
\begin{itemize}
  \item $Q_{\O_{\mathbb{C}_p}}(\Gamma)$ denotes the  quotient field of $\La_{\O_{\mathbb{C}_p}}(\Gamma),$
  \item for $A$ either  $\La_{\O_{\mathbb{C}_p}}(\Gamma)^\times$ or $Q_{\O_{\mathbb{C}_p}}(\Gamma)^\times$ we denote by ${\Hom}_{H,R_\Gamma}(R(G),A)$   the both $H$-invariant and
$R(\Gamma)$-twist-invariant homomorphisms.  The latter means that
$f(\rho\otimes \chi)=\mathrm{tw}_{\chi^{-1}}(f(\rho))$ for all
$\chi$ in $R(\Gamma);$ here $\chi$ is considered as element in
$R(G)$ via the fixed surjection $G\twoheadrightarrow\Gamma$ and
$\mathrm{tw}_\chi:A\to A$ is induced by the action $\gamma\mapsto
\chi(\gamma^{-1})\gamma$ on $ \La_{\O_{\mathbb{C}_p}}(\Gamma),$
compare with \cite[thm.\ 8]{ritter-weiss},

  \item   - analogously as for Det - the homomorphisms $\det$ are induced from the ring
homomorphisms

\[\La_D(G)\to M_{n_\rho}(\O_{\mathbb{C}_p})\widehat{\otimes}_D\La_D(\Gamma)\cong M_{n_\rho}(\La_{\O_{\mathbb{C}_p}}(\Gamma)),\; g\mapsto \rho(g)\otimes \bar{g},\]
and its localisation at $S^*$
\[\La_D(G)_{S^*}\to   M_{n_\rho}(\La_{\O_{\mathbb{C}_p}}(\Gamma))_{S^*}\cong M_{n_\rho}(Q_{\O_{\mathbb{C}_p}}(\Gamma)),\]
respectively.
\item the homomorphism $ev:{\Hom}_{H,R_\Gamma}(R(G),\La_{\O_{\mathbb{C}_p}}(\Gamma)^\times) \to
{\Hom_H}(R(G),\O_{\mathbb{C}_p}^\times),$ which is induced by the
augmentation map
$\La_{\O_{\mathbb{C}_p}}(\Gamma)^\times\to\O_{\mathbb{C}_p}^\times,$
i.e.\ $ev(f)(\rho)=f(\rho)(\chi_{triv}),$ where the latter means
evaluation at the trivial character $\chi_{triv}$ of $\Gamma,$ is
injective by the Weierstrass preparation theorem and the twist
invariance of $f$ (in the kernel of $ev$):
\[f(\rho)(\chi)=\mathrm{tw}_{\chi^{-1}}(f(\rho))(\chi_{triv})=f(\rho\otimes
\chi)(\chi_{triv})=1\] for all $\chi,$ whence $f(\rho)=1$ for all
$\rho,$
\item we do not know whether (need that) the map ${\Hom}_{H,R_\Gamma}(R(G),Q_{\O_{\mathbb{C}_p}}(\Gamma)^\times)   \to
{\mathrm{Maps}}(\mathrm{Irr}G,\Co_p\cup\{\infty\})$ is injective.
\end{itemize}

Note that

\begin{itemize}
  \item a similar diagram exists with $\O$-coefficients (instead of $D$) and $H$ replaced by $G_L$
  and that by construction  the morphisms $\det$ commute with the canonical change of coefficients
  maps.
  \item the composition of the top-line of the above diagram equals  Det.
  \item the image of $\L$ in $\mathrm{Maps}(\mathrm{Irr}G,\Co_p\cup\{\infty\})$ is the map which
  attaches to $\rho$ the value $\L(\rho)$ of $\L$ at $\rho$ and the map into this target is
  multiplicative at least in the following sense: if $f(\rho)\neq 0,\infty$ for all $\rho,$ then
  $(gf)(\rho) = g(\rho)f(\rho)$ for all $\rho$ (with $\infty\cdot a= a\cdot \infty=\infty$ and
  $0\cdot a=a\cdot 0=0$ for $a\neq 0,\infty).$
\end{itemize}


By $\iota$ we denote the canonical map $K_1(\Lambda_\O(G)_{S^*})\to
K_1(\Lambda_D(G)_{S^*}).$

\begin{thm}\label{L-descent}
Assume that
\begin{enumerate}
  \item $\O$ is absolutely unramified,
  \item $L(\rho)/L$ is totally ramified (or trivial) for all $\rho\in \mathrm{Irr}G,$
  \item $\mathcal{L}\in K_1(\Lambda_D(G)_{S^*})$ is
induced from an element in $\La_D(G)\cap
(\Lambda_D(G)_{S^*})^\times$
  \item $\mathcal{L}$ satisfies \[\L(\rho)\in
L(\rho) \mbox{ for all } \rho\in \mathrm{Irr}G,\]
\item there is an $F\in
K_1(\Lambda_\O(G)_{S^*})$ such that $\partial (\L\cdot
\iota(F)^{-1})=0$ (e.g.\ if $\L$ is the characteristic element of
the base change from a module in $\M_{\O,H}(G)$).
\end{enumerate}
     Then there exists
$\L'\in K_1(\Lambda_\O(G)_{S^*})$ with
\[\L'(\rho)=\L(\rho)\mbox{ for all } \rho\in \mathrm{IrrG}\] and \[\partial(\L')=\partial(F),\] in particular \[\iota\partial(\L')=\partial(\L).\]
\end{thm}

\begin{proof}
Consider the following commutative diagram with exact rows
\[\xymatrix{
  K_1(\La_D(G))   \ar[r]^{\pi_D} & K_1(\Lambda_D(G)_{S^*})   \ar[r]^{\partial_D} & K_0(\M_{D,H}(G))   \ar[r]^{ } & K_0(\La_D(G))   \\
  K_1(\La_\O(G)) \ar[r]^{\pi_\O } \ar[u]^{\iota}& K_1(\Lambda_\O(G)_{S^*})  \ar[r]^{\partial_\O} \ar[u]^{\iota}& K_0(\M_{\O,H}(G))\ar[r]^{ } \ar[u]^{\iota}& 0.   }\]

By assumption there exists $\mathcal{D}\in K_1(\La_D(G))$ such that
$\pi_D(\mathcal{D})=\L\cdot \iota(F)^{-1}.$ By Lemma \ref{field-def}
below $\mathrm{Det}(\mathcal{D})$ belongs to
$\mathrm{Det}(K_1(\La_D(G)))^{\Delta}=\mathrm{Det}(K_1(\La_\O(G))),$
i.e.\ there is  a $\mathcal{D}'\in K_1(\La_\O(G))$ such that
$\mathcal{D}(\rho)=\mathcal{D}'(\rho)$ for all $\rho$ in $R(G).$
Setting $\L':=\pi_\O(F)\mathcal{D}'$ and recalling that
$\mathcal{D}'(\rho)= \mathcal{D}(\rho)\neq 0,\infty$ we calculate
\begin{eqnarray*}
{\L}'(\rho)&=& F(\rho) \mathcal{D}'(\rho) \\
         &=& F(\rho) \mathcal{D}(\rho) \\
         &=& \det(F)(\rho)(\chi_{triv})\cdot \det(\mathcal{D})(\rho)(\chi_{triv})\\
         &=& \Big(\det(F)(\rho)\det(\mathcal{D})(\rho)\Big)(\chi_{triv})\\
         &=& \det\big(\iota(F)\pi(\mathcal{D})\big)(\rho)(\chi_{triv})\\
         &=& \det(\L)(\rho)(\chi_{triv})\\
         &=& {\L} (\rho),
\end{eqnarray*}
whence the theorem is proven.
\end{proof}

\begin{lem}\label{field-def}
With notation as in the previous proof we have
\begin{enumerate}
   \item $\det(\mathcal{D})(\rho)\in \O(\rho)[[\Gamma]]$ for all $\rho\in R(G),$
   \item $\mathcal{D}(\rho)=\mathrm{Det}(\mathcal{D})(\rho)\in \O(\rho)$ for all $\rho\in R(G),$
   \item $\mathrm{Det}(\mathcal{D})\in \mathrm{Det}( K_1(\Lambda_D(G) )  )^\Delta.$
 \end{enumerate}
\end{lem}

\begin{proof}
By construction and assumption ($\L$ being induced from ....) we
have $\det(\L)(\rho)\in D(\rho)[[\Gamma]]$ with
$D(\rho):=(\O_{\mathbb{C}_p})^{H_\rho}$ and thus by Lemma \ref{fix}
$\det(\L)(\rho) $ belongs to $\O(\rho) [[\Gamma]].$ On the other
hand $\det(F)(\rho)$ belongs to $Q_{\O(\rho)}(\Gamma)^\times,$
whence \[\det(\mathcal{D})(\rho) \in Q_{\O(\rho)}(\Gamma)^\times\cap
\La_{D(\rho)}(\Gamma)^\times.\] We claim that this intersection
equals $\La_{\O(\rho)}(\Gamma)^\times.$ One inclusion being obvious
we assume that $q$ belongs to the intersection. It follows
immediately that $q(\chi)$ belongs to
$\O(\rho,\chi):=(\O_{\mathbb{C}_p})^{G_{L,\rho}\cap G_{L,\chi}}$ for
all $\chi\in R(\Gamma).$ Thus the claim follows Lemma \ref{fix}
(with $\O(\rho)$  for the base ring $\O$). This proves (i) and as
$\mathcal{D}(\rho)=\det(\mathcal{D})(\rho)(\chi_{triv})$ also (ii)
follows.  In order to show (iii) observe first that $\Delta$ acts
trivially on $\mathrm{Irr}G/G_{\tilde{L}}=\mathrm{Irr}G/G_{{L}}. $
Hence from the following commutative diagram the statement is clear:

\xymatrix{
   {\Hom_H(R(G),\O_{\mathbb{C}_p}^\times)} \ar@{=}[rr]& & {\prod_{\rho\in (\mathrm{Irr}G)/G_L}D(\rho)^\times }  \\
{\Hom_H(R(G),\O_{\mathbb{C}_p}^\times)^\Delta}\ar@{=}[r]\ar@{^(->}[u]
& {\Hom_{G_L}(R(G),\O_{\mathbb{C}_p}^\times)} \ar@{=}[r] &
{\prod_{\rho\in (\mathrm{Irr}G)/G_L}\O(\rho)^\times.}  \ar@{^(->}[u]
 }

\end{proof}

%

\def\Dbar{\leavevmode\lower.6ex\hbox to 0pt{\hskip-.23ex \accent"16\hss}D}
  \def\cfac#1{\ifmmode\setbox7\hbox{$\accent"5E#1$}\else
  \setbox7\hbox{\accent"5E#1}\penalty 10000\relax\fi\raise 1\ht7
  \hbox{\lower1.15ex\hbox to 1\wd7{\hss\accent"13\hss}}\penalty 10000
  \hskip-1\wd7\penalty 10000\box7}
  \def\cftil#1{\ifmmode\setbox7\hbox{$\accent"5E#1$}\else
  \setbox7\hbox{\accent"5E#1}\penalty 10000\relax\fi\raise 1\ht7
  \hbox{\lower1.15ex\hbox to 1\wd7{\hss\accent"7E\hss}}\penalty 10000
  \hskip-1\wd7\penalty 10000\box7} \def\Dbar{\leavevmode\lower.6ex\hbox to
  0pt{\hskip-.23ex \accent"16\hss}D}
  \def\cfac#1{\ifmmode\setbox7\hbox{$\accent"5E#1$}\else
  \setbox7\hbox{\accent"5E#1}\penalty 10000\relax\fi\raise 1\ht7
  \hbox{\lower1.15ex\hbox to 1\wd7{\hss\accent"13\hss}}\penalty 10000
  \hskip-1\wd7\penalty 10000\box7}
  \def\cftil#1{\ifmmode\setbox7\hbox{$\accent"5E#1$}\else
  \setbox7\hbox{\accent"5E#1}\penalty 10000\relax\fi\raise 1\ht7
  \hbox{\lower1.15ex\hbox to 1\wd7{\hss\accent"7E\hss}}\penalty 10000
  \hskip-1\wd7\penalty 10000\box7}
\providecommand{\bysame}{\leavevmode\hbox
to3em{\hrulefill}\thinspace}
\providecommand{\MR}{\relax\ifhmode\unskip\space\fi MR }
\providecommand{\MRhref}[2]{%
  \href{http://www.ams.org/mathscinet-getitem?mr=#1}{#2}
} \providecommand{\href}[2]{#2}

\end{document}